\documentclass[11pt]{elsarticle}

\usepackage{amsmath,amssymb,amsthm}
\usepackage{enumerate}
\usepackage[mathscr]{euscript}

\usepackage{tikz}
\usepackage{tikz-cd}
\usetikzlibrary{matrix, calc, arrows}
\usepackage{latexsym,amscd,epsfig,verbatim}

\newtheorem{thm}{Theorem}[section]
\newtheorem{lem}[thm]{Lemma}
\newtheorem{prop}[thm]{Proposition}
\newtheorem{cor}[thm]{Corollary}
\newtheorem{conj}[thm]{Conjecture}
\theoremstyle{definition}
\newtheorem{defn}[thm]{Definition}
\newtheorem{note}[thm]{Notation}
\newtheorem{ex}[thm]{Example}
\newtheorem{rem}[thm]{Remark}
\newcommand\V{V(\Gamma)}

\makeatletter
\def\ps@pprintTitle{%
 \let\@oddhead\@empty
 \let\@evenhead\@empty
 \def\@oddfoot{}%
 \let\@evenfoot\@oddfoot}
\makeatother

\begin{document}

\title{Artin Group Presentations Arising from Cluster Algebras}
\author{Jacob Haley \fnref{haley}}
\address{University of Notre Dame, Department of Mathematics\\
Notre Dame IN 46556}

\author{David Hemminger \fnref{hemminger}}
\address{Duke University, Mathematics Department \\
Box 90320, Durham NC 27708-0320}

\author{Aaron Landesman \fnref{landesman}}
\address{Harvard University, Department of Mathematics\\
One Oxford Street, Cambridge MA 02138}

\author{Hailee Peck \fnref{peck}}
\address{Millikin University, Department of Mathematics \\
1184 W Main St., Decatur IL 62522}

\fntext[haley]{jachaley@umich.edu}
\fntext[hemminger]{dhemminger@math.ucla.edu}
\fntext[landesman]{aaronlandesman@college.harvard.com}
\fntext[peck]{hpeck@millikin.edu}

\date{\today}
\begin{keyword} Artin group \sep
cluster algebra \sep
diagram mutations
\end{keyword}

\begin{abstract}
In 2003, Fomin and Zelevinsky proved that finite type cluster algebras can be classified by Dynkin diagrams. Then in 2013, Barot and Marsh defined the presentation of a reflection group associated to a Dynkin diagram in terms of an edge-weighted, oriented graph, and proved that this group is invariant (up to isomorphism) under diagram mutations. In this paper, we extend Barot and Marsh's results to Artin group presentations, defining new generator relations and showing mutation-invariance for these presentations.
\end{abstract}

\maketitle

\section{Introduction \& Motivation}
\label{sec:Intro}

In \cite{FZ02}, Fomin and Zelevinsky first introduced the concept of cluster algebras. Barot and Marsh extended Fomin and Zelevinsky's results in \cite{BM13}, providing a presentation of the reflection group associated to a Dynkin diagram with generators that correspond to elements of a companion basis associated to a seed of a finite type cluster algebra. They define generator relations corresponding to chordless cycles arising in diagrams of finite type in order to give a Coxeter group presentation for these diagrams. They also proved that this Coxeter group presentation is invariant up to isomorphism under the mutation equivalence relation. That is, given a diagram $\Gamma$ and a diagram mutation equivalent to $\Gamma$, denoted $\Gamma^{\prime} = \mu_k(\Gamma)$, they proved that $W_{\Gamma} \cong W_{\Gamma^{\prime}}$, where $W_{\Gamma}$ and $W_{\Gamma^{\prime}}$ are the group presentations corresponding to $\Gamma$ and $\Gamma^{\prime}$, respectively.

In our paper, we define $A_{\Gamma}$ to be the Artin group presentation arising from a cluster algebra, where $\Gamma$ is the diagram associated to the cluster algebra. We provide the necessary relations for the generators of the group, and show that these relations hold under mutations of vertices in a diagram. Our main result is to show that this Artin group presentation is invariant, up to isomorphism, under the mutation equivalence relation. We state the result here, but present the detailed proof in Section \ref{sec:main_result}.

\begin{thm}[Theorem \ref{thm:main}]
Let $\Gamma$ be a diagram of finite type, and let $\Gamma^{\prime} = \mu_k(\Gamma)$ be the mutation of $\Gamma$ at vertex $k$. Then $A_{\Gamma} \cong A_{\Gamma^{\prime}}$.
\end{thm}

\indent Section \ref{sec:main_definitions} provides the necessary definitions and fundamental results from \cite{BM13} to motivate our own results. For further definitions and references on the topic, we refer the reader to \cite{FZ02}. Section \ref{sec:finite-type_diagrams} will review theory from \cite{FZ02}, \cite{FZ03} as well as review the classifications (from \cite{BM13}) of mutations of diagrams and their oriented chordless cycles. In Section \ref{sec:defn_artingroup}, we define the appropriate relations for our Artin group presentations. Section \ref{sec:one_relation} specifies how certain relations in chordless cycles imply other relations in those chordless cycles. Section \ref{sec:main_result} will provide the proof that the Artin group defined for a diagram $\Gamma$ is invariant up to isomorphism under mutations of $\Gamma$. Finally, in Section \ref{sec:affine}, we will state a conjecture which could extend the main result to diagrams of affine type.\\

In the process of writing this paper, we became aware of related results obtained independently by Grant and Marsh in \cite{GM14} and Qiu in \cite{Q14}. Using different approaches, the two papers found relations for the Artin group in simply-laced diagrams. However, their relations differ from the ones used in our paper, and we show that our results hold for all diagrams of finite type. Furthermore, our approach to the problem is more combinatorial than the topological approaches found in \cite{Q14} and \cite{GM14}.

\section{Background}
\label{sec:main_definitions}

We begin by introducing some preliminary notations and definitions which will aid the reader in understanding the results in the following sections. For further references on cluster algebras, we refer the reader to \cite{FZ02, FZ03}, and for a more detailed description of Artin group presentations, we direct attention to \cite{FN61}. We also provide references to several lemmas and propositions from \cite{BM13} which were helpful in formulating our own results. \\

The initial introduction of cluster algebras by Fomin and Zelevinsky was aimed at making further strides in the areas of representation theory, Lie theory, and total positivity. Since then, the study of cluster algebras has provided a motivation for applications in various other areas of mathematics, including quiver representations. Of particular interest were \textit{finite type} cluster algebras, those with a finite number of distinguished generators. In the sequel to their introductory paper (\cite{FZ03}), Fomin and Zelevinsky introduce the concept of \textit{mutation equivalence} between diagrams, proving that a connected graph is mutation equivalent to an oriented Dynkin diagram if and only if all mutation equivalent graphs have edge weights not exceeding 3. In particular, this proves that finite type cluster algebras can be classified by Dynkin diagrams. \\

A \textit{cluster algebra} is an integral domain which can be generated by a set of elements called \textit{cluster variables} that satisfy certain exchange relations. Following the style of \cite{FZ02} and \cite{BM13}, we will define cluster algebras in terms of \textit{skew-symmetrisable} matrices (that is, a matrix $B$ such that there exists a diagonal matrix $D$ of the same size with $D_{ii} >0$ such that $DB$ is skew-symmetric). Let $\mathbb{F} = \mathbb{Q}(u_1, u_2, \ldots, u_n)$ be the field of rational functions in $n$ indeterminates over $\mathbb{Q}$. We will define an \textit{initial seed} for the cluster algebra to be a fixed pair $(\textbf{x}, B)$, where $\textbf{x} = \{x_1, \ldots, x_n\}$ is a free generating set of $\mathbb{F}$ and $B$ is an $n \times n$ skew-symmetric matrix. Define $x_k^{\prime} \in \mathbb{F}$ by the \textit{exchange relation}
\begin{displaymath}
x_k^{\prime}x_k = \prod_{B_{ik} > 0}{x_i^{B_{ik}}} + \prod_{B_{ik} < 0}{x_i^{-B_{ik}}}
\end{displaymath}
Then, given an initial seed $(\textbf{x}, B)$ and $k \in {1,2,\ldots,n}$, we can define a \textit{mutation} of the seed at $k$, denoted $\mu_k(\boldsymbol{x}, B) = (\textbf{x}^{\prime}, B^{\prime})$ where:
\begin{displaymath}
B_{ij}^{\prime} = \begin{cases} - B_{ij} & \mbox{ if } i = k \mbox{ or } j = k;\\
B_{ij} + \frac{|B_{ik}|B_{kj} + B_{ik}|B_{kj}|}{2} & \mbox{ otherwise. }\\
\end{cases}
\end{displaymath}
and $\textbf{x}^{\prime} = \left\{ x_1, x_2, \ldots, x_{k-1}, x_k^{\prime}, x_{k+1}, \ldots, x_n \right\}$.
Such a mutation or a sequence of such mutations generate \textit{seeds} which in turn generate all cluster variables in that, for each $\textbf{x} = \left\{ x_1, \ldots, x_n \right\}$ corresponding to a seed of the cluster algebra, the entries $x_i$ are the cluster variables. \\

 A cluster algebra is said to be of \textit{finite type} if the number of cluster variables that generate it is finite (if it has finitely many seeds). For each finite type cluster algebra, we can associate to its corresponding skew-symmetrisable matrix an edge-weighted, oriented graph, called a \textit{diagram}. We will often denote this diagram by $\Gamma$, and the vertex set of $\Gamma$ by $\V$. We will denote two connected vertices by $i \rightarrow j$, or by $i\--\ j$ if the orientation is not specified. The diagram is determined by the following: for $i, j \in \V$, $i \xrightarrow{w} j$ if and only if $B_{ij} > 0$ and $w = |B_{ij}B_{ji}|$ is the weight of the edge. A skew-symmetrisable matrix $B$ is \textit{2-finite} if $|B_{ij}B_{ji}| \leq 3$ for $i, j \in \left\{ 1, \ldots, n \right\}$. By \cite[7.5]{FZ02}, we have that if $B$ is 2-finite, all 3-cycles in the unoriented graph underlying our diagram must be oriented cyclically. \\

Just as we can define mutations of the seeds of a cluster variable, we can also define mutations of a diagram associated to a cluster algebra of finite type by the following set of rules:
\begin{prop}\cite{FZ03}
Let $B$ be a 2-finite skew-symmetrisable matrix. Then $\Gamma(\mu_k(B))$ is uniquely determined by $\Gamma(B)$ as follows:
\begin{itemize}
\item Reverse the orientations of all edges in $\Gamma(B)$ incident to $k$ (leaving the weights unchanged)
\item For any path in $\Gamma(B)$ of form $i \xrightarrow{a} k \xrightarrow{b} j$ (i.e. with $a,b$ positive), let $c$ be the weight on the edge $j \rightarrow i$, taken to be zero if there is no such arrow. Let $c'$ be determined by $c'\geq 0$ and
$\pm \sqrt{c} \pm \sqrt{c'} = \sqrt{ab}$, 
where the sign before $\sqrt{c}$ (respectively, $\sqrt{c'}$) is positive if the arrows form an oriented cycle and negative otherwise. Then $\Gamma(B)$ changes as in Figure \ref{mutation}, taking the case $c' = 0$ to mean no arrow between $i$ and $j$.
\end{itemize}
\end{prop}

\begin{figure}
\begin{center}
\begin{tikzpicture}[baseline=-0.5ex]

\matrix[matrix of math nodes,column sep={25pt,between origins},row
    sep={40pt,between origins}] at (0,0)
  {
    && |[name = u]|\bullet& \\
    &|[name=dl]| \bullet & &|[name=dr]| \bullet \\
  };
\draw[->]
  (u) to node[right] {b} (dr)
  ;
  \draw[->]
  (dl) to node[left]{a}(u)
  ;
  \draw[->]
  (dr) to node[below]{c}(dl)
  ;
  \node [yshift = .4 cm] at (u){k};
    \node [yshift = -.4 cm,xshift = .4 cm] at (dr){j};
    \node [yshift = -.4 cm,xshift = -.4 cm] at (dl){i};
    \matrix[matrix of math nodes,column sep={25pt,between origins},row
    sep={40pt,between origins}] at (4,0)
  {
    && |[name = u]|\bullet& \\
    &|[name=dl]| \bullet & &|[name=dr]| \bullet \\
  };
\draw[->]
  (dr)to node[right] {b} (u)
  ;
  \draw[->]
  (u) to node[left]{a}(dl)
  ;
  \draw[->]
  (dl) to node[below]{c'}(dr)
  ;
  \node [yshift = .4 cm] at (u){k};
    \node [yshift = -.4 cm,xshift = .4 cm] at (dr){j};
    \node [yshift = -.4 cm,xshift = -.4 cm] at (dl){i};
 
\end{tikzpicture}
\end{center}
\caption{Mutation of $\Gamma$ at a node k}\label{mutation}
\end{figure}
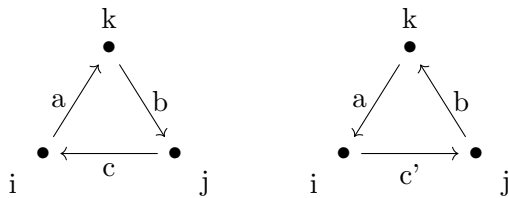

\begin{note}
We notate this mutation of $\Gamma(B)$ at vertex $k$ by $\mu_k(\Gamma)$. \\
\end{note}

\begin{note}
\label{defn:mij}
\indent Given a diagram $\Gamma$, Barot and Marsh define for $i, j \in \V$,
\begin{displaymath}
m_{ij} = \begin{cases} 2 & \mbox{if } i \mbox{ and } j \mbox{ are not connected;} \\
3 & \mbox{if } i \mbox{ and } j \mbox{ are connected by an edge of weight } 1;\\
4 & \mbox{if } i \mbox{ and } j \mbox{ are connected by an edge of weight } 2;\\
6 & \mbox{if } i \mbox{ and } j \mbox{ are connected by an edge of weight } 3.\\
\end{cases}
\end{displaymath}	
\end{note}

Then, they define $W(\Gamma)$ to be the group generated by $s_i$, for $i \in \V$, under the following relations. Note that $e$ will denote the identity element of $W(\Gamma)$.
\begin{enumerate}
\item $s_i^2 = e$ for all $i$;
\item $(s_is_j)^{m_{ij}} = e$ for all $i \neq j$;
\item For any chordless cycle (as defined in Definition \ref{chordlesscycle}) $C$ in $\Gamma$, where
\begin{displaymath}
C = i_0 \xrightarrow{w_1} i_1 \xrightarrow{w_2} \cdots \xrightarrow{w_{d-1}} i_{d-1} \xrightarrow{w_0} i_0
\end{displaymath}
and all of the weights $w_k$ are $1$ or $w_0 = 2$, we have
\begin{displaymath}
(s_{i_0}s_{i_1}\cdots s_{i_{d-2}}s_{i_{d-1}}s_{i_{d-2}}\cdots s_{i_1})^2 =e.
\end{displaymath}
\end{enumerate}

Using this group presentation, Barot and Marsh state the following result:
\begin{thm}\label{thm:barot_and_marsh_A}\cite[Theorem A]{BM13}
Let $\Gamma$ be the diagram associated to a seed in a cluster algebra of finite type. Then $W(\Gamma)$ is isomorphic to the corresponding reflection group.
\end{thm}

In Section 3 of \cite{BM13}, Barot and Marsh provide an alteration of the group $W(\Gamma)$ in order to extend the group definition to any diagram of finite type. More specifically, they provide relations such that $W_{\Gamma}$, as defined below, is isomorphic to $W(\Gamma)$ (\cite[Proposition 4.5]{BM13}).

\begin{defn}
Let $W_{\Gamma}$ be the group with generators $s_i, i = 1,2,\ldots, n$, subject to the following relations:
\begin{itemize}
\item{(R1)} $s_i^2 = e$ for all $i$
\item{(R2)} $\left(s_is_j\right)^{m_{ij}} = e$ for all $i \neq j$
\end{itemize}
Furthermore, for a chordless cycle $C : i_0 \rightarrow i_1 \rightarrow \cdots \rightarrow i_{d-1} \rightarrow i_0$ and for $a = 0,1,2,\ldots, d-1$, define \textbf{$r\left(i_a, i_{a+1}\right) = s_{i_a}s_{i_{a+1}} \cdots s_{i_{a+d-1}}s_{i_{a+d-2}} \cdots s_{i_{a+1}}$}.\\

\vspace{0.1cm}
Then we have the following relations:
\begin{itemize}
\item{(R3)(a)} If all the weights in the edges of $C$ are $1$, then $r(i_a, i_{a+1})^2 = e$
\item{(R3)(b)} If $C$ has some edges of weight $2$, then $r(i_a, i_{a+1})^k = e$ where $k = 4-w_a$ and $w_a$ is the weight of the edge $i_a\--\ i_{a-1}$
\end{itemize}
\end{defn}

Defining the group $W_{\Gamma}$ with relations as shown above allows them to prove certain characteristics of the interaction between the relations in this group for the chordless cycles underlying the diagrams in question. In particular, they prove the following result.
\begin{thm}\cite[Theorem 5.4a]{BM13}
Let $\Gamma$ be a diagram of finite type and $\Gamma^{\prime} = \mu_k(\Gamma)$ the mutation of $\Gamma$ at vertex $k$. Then $W_{\Gamma} \cong W_{\Gamma^{\prime}}$.
\end{thm}

The rest of the paper will be devoted to building up analogous relations, defined in Definition~\ref{grp def} to prove a similar result in the case of Artin groups. For $\Gamma$ a diagram of finite type, we define the Artin group associated to $\Gamma$ as in Section \ref{sec:defn_artingroup}. We will then use the group relations presented in this definition to prove the mutation invariance of $A_{\Gamma}$ in Section \ref{sec:main_result}.

\section{Diagrams of Finite Type}
\label{sec:finite-type_diagrams}

In this section, we shall review the structure of diagrams of finite type, and how their cycles are affected by mutation. This section is simply a recap of \cite[Section 2]{BM13}. First, in Proposition ~\ref{cycle_types}, all types of chordless cycles in diagrams of finite type are classified. Second, in Corollary ~\ref{local_three_picture} all possible local pictures between a mutated vertex and two adjacent vertices are drawn. Finally, in Lemma ~\ref{lem:chordless_cycles}, all chordless cycles introduced from a mutation are drawn. These three lemmas will be crucial in proving the main result Theorem ~\ref{thm:main}, as they will allow us to inspect precisely which relations are added and removed after mutating at a prescribed vertex.

\begin{defn}
\label{chordlesscycle}
A {\it chordless cycle} of an unoriented graph $G$ is a connected subgraph $H \subset G$ such that the number of vertices in $H$ is equal to the number of edges in $H,$ and the edges in $H$ form a single cycle.
\end{defn}

\begin{prop}
\label{cycle_types}
\cite[Proposition 2.1]{BM13} Let $\Gamma$ be a diagram of finite type. Then, a chordless cycle in the unoriented graph of $\Gamma$ is cyclically oriented in $\Gamma$. Furthermore, the unoriented graph underlying the cycle must either be a cycle such that all edges have weight 1, a square with two opposite edges of weight 2 and two opposite edges of weight 1, or a triangle with two edges of weight 2 and one of weight 1, as pictured in Figure \ref{fig:cycle_types}.
\end{prop}

\begin{figure}[h]
\hspace*{-1.2cm}
\begin{tikzpicture}[baseline=-0.5ex]
 \matrix[matrix of math nodes,column sep={40pt,between origins},row
    sep={40pt,between origins}] at (0,0)
  {
    & &|[name = ul]|\circ& |[name = ur]|\circ&\\
    &|[name=lu]| \circ &&&|[name=ru]| \circ \\
     & |[name = ld]|\circ&&& |[name = rd]|\circ\\
    &&|[name=dl]| \circ &|[name=dr]| \circ &\\
  };
\draw
  (ul) -- node[above] {1} (ur)
  ;
  \draw
  (ul) -- node[left]{1}(lu)
  ;
  \draw
  (dr) -- node[right]{1}(rd)
  ;
  \draw
  (lu) -- node[left]{1}(ld)
  ;
  \draw
  (ru) -- node[right]{1}(rd)
  ;
  \draw
  (ur) -- node[right]{1}(ru)
  ;
  \draw
  (dl) -- node[left]{1}(ld)
  ;
  \draw[dotted]
  (dl) -- (dr)
  ;
 \matrix[matrix of math nodes,column sep={40pt,between origins},row
    sep={40pt,between origins}] at (5,0)
  {
    & |[name = ul]|\circ& |[name = ur]|\circ\\
    &|[name=dl]| \circ &|[name=dr]| \circ \\
  };
\draw
  (ul) -- node[above] {1} (ur)
  ;
  \draw
  (dl) -- node[left]{2}(ul)
  ;
  \draw
  (dr) -- node[below]{1}(dl)
  ;
  \draw
  (ur) -- node[right]{2} (dr)
  ;
   \matrix[matrix of math nodes,column sep={25pt,between origins},row
    sep={40pt,between origins}] at (9,0)
  {
    && |[name = u]|\circ& \\
    &|[name=dl]| \circ & &|[name=dr]| \circ \\
  };
\draw
  (u) -- node[right] {2} (dr)
  ;
  \draw
  (dl) -- node[left]{2}(u)
  ;
  \draw
  (dr) -- node[below]{1}(dl)
  ;
  
   \end{tikzpicture}
   
\caption{Possible chordless cycles in a diagram of finite type, as found in \cite[Figure 2]{BM13}} \label{fig:cycle_types}
\end{figure}
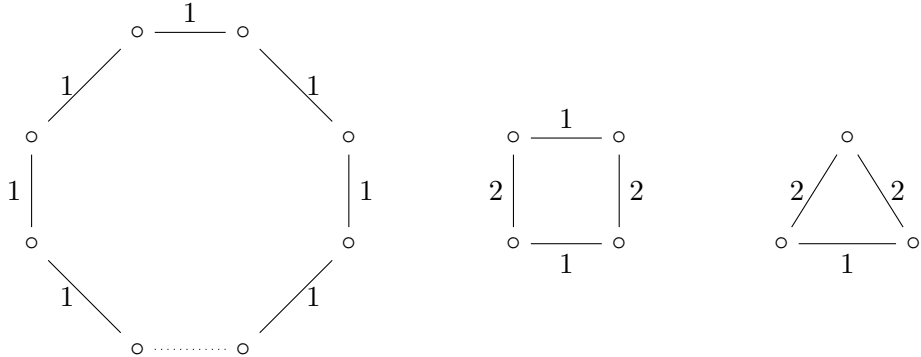

\begin{cor}
\label{local_three_picture}
\cite[Corollary 2.3]{BM13} Let $\Gamma$ be a graph of finite type and suppose there are three vertices, labeled $i,j,k$ with both $i,j$ connected to $k.$ Then mutation at $k$ on the induced subdiagram appear as in Figure \ref{fig:local_three}, either from left to right or right to left, up to switching $i$ and $j,$
\end{cor}

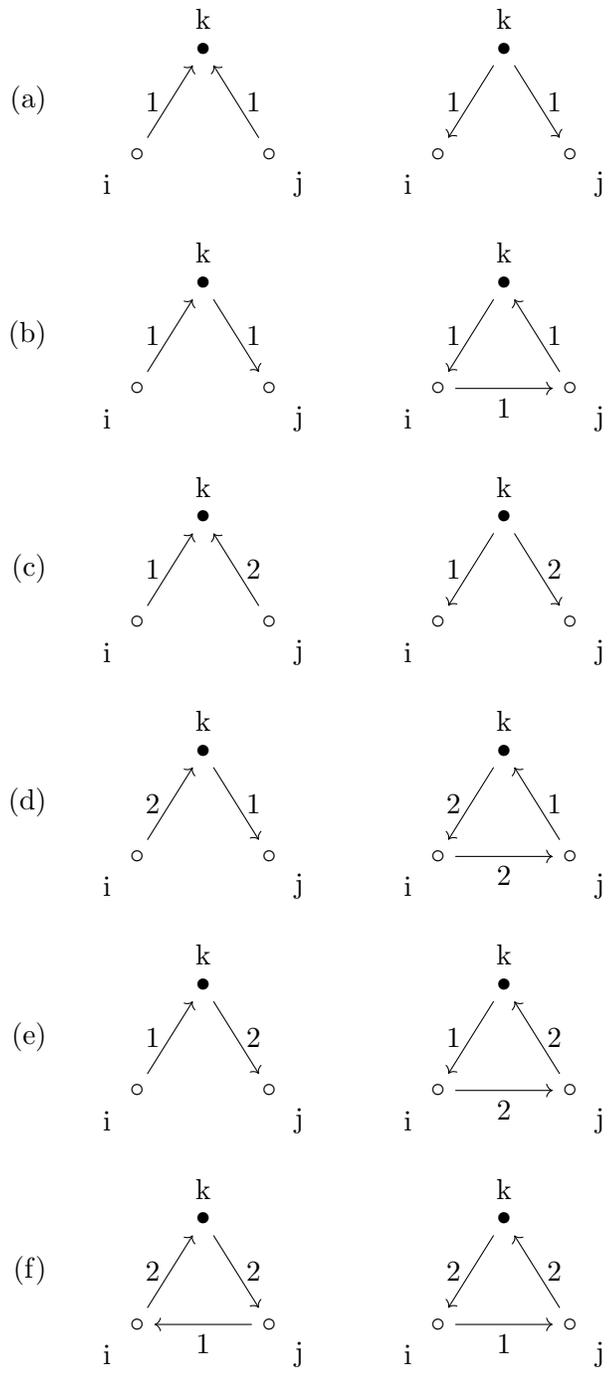
\begin{figure}
\begin{enumerate}[(a)]
\item
\begin{tikzpicture}[baseline=-0.5ex]
\matrix[matrix of math nodes,column sep={25pt,between origins},row
    sep={40pt,between origins}] at (0,0)
  {
    && |[name = u]|\bullet& \\
    &|[name=dl]| \circ & &|[name=dr]| \circ \\
  };
\draw[->]
  (dr) to node[right] {1} (u)
  ;
  \draw[->]
  (dl)to node[left]{1}(u)
  ;
  \node [yshift = .4 cm] at (u){k};
    \node [yshift = -.4 cm,xshift = .4 cm] at (dr){j};
    \node [yshift = -.4 cm,xshift = -.4 cm] at (dl){i};
    \matrix[matrix of math nodes,column sep={25pt,between origins},row
    sep={40pt,between origins}] at (4,0)
  {
    && |[name = u]|\bullet& \\
    &|[name=dl]| \circ & &|[name=dr]| \circ \\
  };
\draw[->]
   (u)to node[right] {1} (dr)
  ;
  \draw[->]
  (u) to node[left]{1}(dl)
  ;
  \node [yshift = .4 cm] at (u){k};
    \node [yshift = -.4 cm,xshift = .4 cm] at (dr){j};
    \node [yshift = -.4 cm,xshift = -.4 cm] at (dl){i};
   \end{tikzpicture}
\item
\begin{tikzpicture}[baseline=-0.5ex]
\matrix[matrix of math nodes,column sep={25pt,between origins},row
    sep={40pt,between origins}] at (0,0)
  {
    && |[name = u]|\bullet& \\
    &|[name=dl]| \circ & &|[name=dr]| \circ \\
  };
\draw[->]
  (u) to node[right] {1} (dr)
  ;
  \draw[->]
  (dl)to node[left]{1}(u)
  ;
  \node [yshift = .4 cm] at (u){k};
    \node [yshift = -.4 cm,xshift = .4 cm] at (dr){j};
    \node [yshift = -.4 cm,xshift = -.4 cm] at (dl){i};
    \matrix[matrix of math nodes,column sep={25pt,between origins},row
    sep={40pt,between origins}] at (4,0)
  {
    && |[name = u]|\bullet& \\
    &|[name=dl]| \circ & &|[name=dr]| \circ \\
  };
\draw[->]
  (dr) to node[right] {1}(u)
  ;
  \draw[->]
  (u) to node[left]{1}(dl)
  ;
  \draw[->]
  (dl) to node[below]{1}(dr)
  ;
  \node [yshift = .4 cm] at (u){k};
    \node [yshift = -.4 cm,xshift = .4 cm] at (dr){j};
    \node [yshift = -.4 cm,xshift = -.4 cm] at (dl){i};
   \end{tikzpicture}

\item
\begin{tikzpicture}[baseline=-0.5ex]
\matrix[matrix of math nodes,column sep={25pt,between origins},row
    sep={40pt,between origins}] at (0,0)
  {
    && |[name = u]|\bullet& \\
    &|[name=dl]| \circ & &|[name=dr]| \circ \\
  };
\draw[->]
  (dr) to node[right] {2} (u)
  ;
  \draw[->]
  (dl)to node[left]{1}(u)
  ;
  \node [yshift = .4 cm] at (u){k};
    \node [yshift = -.4 cm,xshift = .4 cm] at (dr){j};
    \node [yshift = -.4 cm,xshift = -.4 cm] at (dl){i};
    \matrix[matrix of math nodes,column sep={25pt,between origins},row
    sep={40pt,between origins}] at (4,0)
  {
    && |[name = u]|\bullet& \\
    &|[name=dl]| \circ & &|[name=dr]| \circ \\
  };
\draw[->]
   (u)to node[right] {2} (dr)
  ;
  \draw[->]
  (u) to node[left]{1}(dl)
  ;
  \node [yshift = .4 cm] at (u){k};
    \node [yshift = -.4 cm,xshift = .4 cm] at (dr){j};
    \node [yshift = -.4 cm,xshift = -.4 cm] at (dl){i};
   \end{tikzpicture}
   \item
\begin{tikzpicture}[baseline=-0.5ex]
\matrix[matrix of math nodes,column sep={25pt,between origins},row
    sep={40pt,between origins}] at (0,0)
  {
    && |[name = u]|\bullet& \\
    &|[name=dl]| \circ & &|[name=dr]| \circ \\
  };
\draw[->]
  (u) to node[right] {1} (dr)
  ;
  \draw[->]
  (dl) to node[left]{2}(u)
  ;
  \node [yshift = .4 cm] at (u){k};
    \node [yshift = -.4 cm,xshift = .4 cm] at (dr){j};
    \node [yshift = -.4 cm,xshift = -.4 cm] at (dl){i};
    \matrix[matrix of math nodes,column sep={25pt,between origins},row
    sep={40pt,between origins}] at (4,0)
  {
    && |[name = u]|\bullet& \\
    &|[name=dl]| \circ & &|[name=dr]| \circ \\
  };
\draw[->]
  (dr)to node[right] {1} (u)
  ;
  \draw[->]
  (u) to node[left]{2}(dl)
  ;
  \draw[->]
  (dl) to node[below]{2}(dr)
  ;
  \node [yshift = .4 cm] at (u){k};
    \node [yshift = -.4 cm,xshift = .4 cm] at (dr){j};
    \node [yshift = -.4 cm,xshift = -.4 cm] at (dl){i};
   \end{tikzpicture}
   \item
\begin{tikzpicture}[baseline=-0.5ex]
\matrix[matrix of math nodes,column sep={25pt,between origins},row
    sep={40pt,between origins}] at (0,0)
  {
    && |[name = u]|\bullet& \\
    &|[name=dl]| \circ & &|[name=dr]| \circ \\
  };
\draw[->]
  (u) to node[right] {2} (dr)
  ;
  \draw[->]
  (dl) to node[left]{1}(u)
  ;
  \node [yshift = .4 cm] at (u){k};
    \node [yshift = -.4 cm,xshift = .4 cm] at (dr){j};
    \node [yshift = -.4 cm,xshift = -.4 cm] at (dl){i};
    \matrix[matrix of math nodes,column sep={25pt,between origins},row
    sep={40pt,between origins}] at (4,0)
  {
    && |[name = u]|\bullet& \\
    &|[name=dl]| \circ & &|[name=dr]| \circ \\
  };
\draw[->]
  (dr) to node[right] {2} (u)
  ;
  \draw[->]
  (u) to node[left]{1}(dl)
  ;
  \draw[->]
  (dl) to node[below]{2}(dr)
  ;
  \node [yshift = .4 cm] at (u){k};
    \node [yshift = -.4 cm,xshift = .4 cm] at (dr){j};
    \node [yshift = -.4 cm,xshift = -.4 cm] at (dl){i};
   \end{tikzpicture}
\item
\begin{tikzpicture}[baseline=-0.5ex]
\matrix[matrix of math nodes,column sep={25pt,between origins},row
    sep={40pt,between origins}] at (0,0)
  {
    && |[name = u]|\bullet& \\
    &|[name=dl]| \circ & &|[name=dr]| \circ \\
  };
\draw[->]
   (u) to node[right] {2}(dr)
  ;
  \draw[->]
  (dl) to node[left]{2}(u)
  ;
  \draw[->]
  (dr) to node[below]{1}(dl)
  ;
  \node [yshift = .4 cm] at (u){k};
    \node [yshift = -.4 cm,xshift = .4 cm] at (dr){j};
    \node [yshift = -.4 cm,xshift = -.4 cm] at (dl){i};
    \matrix[matrix of math nodes,column sep={25pt,between origins},row
    sep={40pt,between origins}] at (4,0)
  {
    && |[name = u]|\bullet& \\
    &|[name=dl]| \circ & &|[name=dr]| \circ \\
  };
\draw[->]
  (dr) to node[right] {2} (u)
  ;
  \draw[->]
  (u) to node[left]{2}(dl)
  ;
  \draw[->]
  (dl) to node[below]{1}(dr)
  ;
  \node [yshift = .4 cm] at (u){k};
    \node [yshift = -.4 cm,xshift = .4 cm] at (dr){j};
    \node [yshift = -.4 cm,xshift = -.4 cm] at (dl){i};
   \end{tikzpicture}
\end{enumerate}

\caption{Local picture of the mutation of a finite type diagram, as found in \cite[Figure 4]{BM13}} \label{fig:local_three}
\end{figure}

\begin{figure}
\begin{enumerate}[(a)]
\item
\begin{tikzpicture}[baseline=-0.5ex]
\matrix[matrix of math nodes,column sep={25pt,between origins},row
    sep={40pt,between origins}] at (0,0)
  {
    && |[name = u]|\bullet& \\
    &|[name=dl]| \circ & &|[name=dr]| \circ \\
  };
\draw[->]
  (dr) to node[right] {1} (u)
  ;
  \draw[->]
  (u) to node[left]{1}(dl)
  ;
  \node [yshift = .4 cm] at (u){k=2};
    \node [yshift = -.4 cm,xshift = .4 cm] at (dr){3};
    \node [yshift = -.4 cm,xshift = -.4 cm] at (dl){1};
    \matrix[matrix of math nodes,column sep={25pt,between origins},row
    sep={40pt,between origins}] at (4,0)
  {
    && |[name = u]|\bullet& \\
    &|[name=dl]| \circ & &|[name=dr]| \circ \\
  };
\draw[->]
  (u) to node[right] {1} (dr)
  ;
  \draw[->]
  (dl) to node[left]{1}(u)
  ;
  \draw[->]
  (dr) to node[below]{1}(dl)
  ;
  \node [yshift = .4 cm] at (u){k=2};
    \node [yshift = -.4 cm,xshift = .4 cm] at (dr){3};
    \node [yshift = -.4 cm,xshift = -.4 cm] at (dl){1};
    \node [yshift = -.9 cm] at (u){$C'$};
   \end{tikzpicture}
   \item
\begin{tikzpicture}[baseline=-0.5ex]
\matrix[matrix of math nodes,column sep={25pt,between origins},row
    sep={40pt,between origins}] at (0,0)
  {
    && |[name = u]|\bullet& \\
    &|[name=dl]| \circ & &|[name=dr]| \circ \\
  };
\draw[->]
  (dr) to node[right] {1} (u)
  ;
  \draw[->]
  (u) to node[left]{2}(dl)
  ;
  \node [yshift = .4 cm] at (u){k=2};
    \node [yshift = -.4 cm,xshift = .4 cm] at (dr){3};
    \node [yshift = -.4 cm,xshift = -.4 cm] at (dl){1};
    \matrix[matrix of math nodes,column sep={25pt,between origins},row
    sep={40pt,between origins}] at (4,0)
  {
    && |[name = u]|\bullet& \\
    &|[name=dl]| \circ & &|[name=dr]| \circ \\
  };
\draw[->]
  (u) to node[right] {1} (dr)
  ;
  \draw[->]
  (dl) to node[left]{2}(u)
  ;
  \draw[->]
  (dr) to node[below]{2}(dl)
  ;
  \node [yshift = .4 cm] at (u){k=2};
    \node [yshift = -.4 cm,xshift = .4 cm] at (dr){3};
    \node [yshift = -.4 cm,xshift = -.4 cm] at (dl){1};
    \node [yshift = -.9 cm] at (u){$C'$};
   \end{tikzpicture}
   \item
\begin{tikzpicture}[baseline=-0.5ex]
\matrix[matrix of math nodes,column sep={25pt,between origins},row
    sep={40pt,between origins}] at (0,0)
  {
    && |[name = u]|\bullet& \\
    &|[name=dl]| \circ & &|[name=dr]| \circ \\
  };
\draw[->]
  (dr) to node[right] {2} (u)
  ;
  \draw[->]
  (u) to node[left]{1}(dl)
  ;
  \node [yshift = .4 cm] at (u){k=2};
    \node [yshift = -.4 cm,xshift = .4 cm] at (dr){3};
    \node [yshift = -.4 cm,xshift = -.4 cm] at (dl){1};
    \matrix[matrix of math nodes,column sep={25pt,between origins},row
    sep={40pt,between origins}] at (4,0)
  {
    && |[name = u]|\bullet& \\
    &|[name=dl]| \circ & &|[name=dr]| \circ \\
  };
\draw[->]
  (u) to node[right] {2} (dr)
  ;
  \draw[->]
  (dl) to node[left]{1}(u)
  ;
  \draw[->]
  (dr) to node[below]{2}(dl)
  ;
  \node [yshift = .4 cm] at (u){k=2};
    \node [yshift = -.4 cm,xshift = .4 cm] at (dr){3};
    \node [yshift = -.4 cm,xshift = -.4 cm] at (dl){1};
    \node [yshift = -.9 cm] at (u){$C'$};
   \end{tikzpicture}
\item
\begin{tikzpicture}[baseline=-0.5ex]
\matrix[matrix of math nodes,column sep={25pt,between origins},row
    sep={40pt,between origins}] at (0,0)
  {
    && |[name = u]|\bullet& \\
    &|[name=dl]| \circ & &|[name=dr]| \circ \\
  };
\draw[->]
  (dr) to node[right] {2} (u)
  ;
  \draw[->]
  (u) to node[left]{2}(dl)
  ;
  \draw[->]
  (dl) to node[below]{1}(dr)
  ;
  \node [yshift = .4 cm] at (u){k=2};
    \node [yshift = -.4 cm,xshift = .4 cm] at (dr){3};
    \node [yshift = -.4 cm,xshift = -.4 cm] at (dl){1};
    \matrix[matrix of math nodes,column sep={25pt,between origins},row
    sep={40pt,between origins}] at (4,0)
  {
    && |[name = u]|\bullet& \\
    &|[name=dl]| \circ & &|[name=dr]| \circ \\
  };
\draw[->]
  (u) to node[right] {2} (dr)
  ;
  \draw[->]
  (dl) to node[left]{2}(u)
  ;
  \draw[->]
  (dr) to node[below]{1}(dl)
  ;
  \node [yshift = .4 cm] at (u){k=2};
    \node [yshift = -.4 cm,xshift = .4 cm] at (dr){3};
    \node [yshift = -.4 cm,xshift = -.4 cm] at (dl){1};
    \node [yshift = -.9 cm] at (u){$C'$};
   \end{tikzpicture}
\item
\begin{tikzpicture}[baseline=-0.5ex]
\matrix[matrix of math nodes,column sep={40pt,between origins},row
    sep={40pt,between origins}] at (0,0)
  {
    & |[name = ul]|\bullet& |[name = ur]|\circ\\
    &|[name=dl]| \circ &|[name=dr]| \circ \\
  };
\draw[->]
  (ur) to node[above] {1} (ul)
  ;
  \draw[->]
  (ul) to node[left]{2}(dl)
  ;
  \draw[->]
  (dr) to node[below]{1}(dl)
  ;
  \draw[->]
  (ur) to node[right]{2} (dr)
  ;
\draw[->]
  (dl) to node[above]{2} (ur)
  ;
    \node [yshift = .4 cm,xshift = -.4 cm] at (ul){k=1};
   \node [yshift = .4 cm,xshift = .4 cm] at (ur){2};
    \node [yshift = -.4 cm,xshift = .4 cm] at (dr){3};
    \node [yshift = -.4 cm,xshift = -.4 cm] at (dl){4};
    \matrix[matrix of math nodes,column sep={40pt,between origins},row
    sep={40pt,between origins}] at (4,0)
  {
    & |[name = ul]|\bullet& |[name = ur]|\circ\\
    &|[name=dl]| \circ &|[name=dr]| \circ \\
  };
\draw[->]
  (ul) to node[above] {1} (ur)
  ;
  \draw[->]
  (dl) to node[left]{2}(ul)
  ;
  \draw[->]
  (dr) to node[below]{1}(dl)
  ;
  \draw[->]
  (ur) to node[right]{2} (dr)
  ;
    \node [yshift = .4 cm,xshift = -.4 cm] at (ul){k=1};
   \node [yshift = .4 cm,xshift = .4 cm] at (ur){2};
    \node [yshift = -.4 cm,xshift = .4 cm] at (dr){3};
    \node [yshift = -.4 cm,xshift = -.4 cm] at (dl){4};
  \node [yshift = -.7 cm,xshift = .7 cm] at (ul){$C'$};
   \end{tikzpicture}
   
\item \begin{tikzpicture}[baseline=-0.5ex]
\matrix[matrix of math nodes,column sep={40pt,between origins},row
    sep={40pt,between origins}] at (0,0)
  {
    & |[name = ul]|\circ& |[name = ur]|\bullet\\
    &|[name=dl]| \circ &|[name=dr]| \circ \\
  };
\draw[->]
  (ur) to node[above] {1} (ul)
  ;
  \draw[->]
  (dl) to node[left]{2}(ul)
  ;
  \draw[->]
  (dr) to node[below]{1}(dl)
  ;
  \draw[->]
  (dr) to node[right]{2} (ur)
  ;
\draw[->]
  (ul) to node[above]{2} (dr)
  ;
    \node [yshift = .4 cm,xshift = -.4 cm] at (ul){1};
   \node [yshift = .4 cm,xshift = .4 cm] at (ur){k=2};
    \node [yshift = -.4 cm,xshift = .4 cm] at (dr){3};
    \node [yshift = -.4 cm,xshift = -.4 cm] at (dl){4};
    \matrix[matrix of math nodes,column sep={40pt,between origins},row
    sep={40pt,between origins}] at (4,0)
  {
    & |[name = ul]|\circ& |[name = ur]|\bullet\\
    &|[name=dl]| \circ &|[name=dr]| \circ \\
  };
\draw[->]
  (ul) to node[above] {1} (ur)
  ;
  \draw[->]
  (dl) to node[left]{2}(ul)
  ;
  \draw[->]
  (dr) to node[below]{1}(dl)
  ;
  \draw[->]
  (ur) to node[right]{2} (dr)
  ;
    \node [yshift = .4 cm,xshift = -.4 cm] at (ul){1};
   \node [yshift = .4 cm,xshift = .4 cm] at (ur){k=2};
    \node [yshift = -.4 cm,xshift = .4 cm] at (dr){3};
    \node [yshift = -.4 cm,xshift = -.4 cm] at (dl){4};
  \node [yshift = -.7 cm,xshift = .7 cm] at (ul){$C'$};
   \end{tikzpicture}

   \item
   \begin{tikzpicture}[baseline=-0.5ex]

    \matrix[matrix of math nodes,column sep={40pt,between origins},row
    sep={40pt,between origins}] at (0,0)
  {
    & |[name = ul]|\bullet& |[name = ur]|\circ\\
    &|[name=dl]| \circ &|[name=dr]| \circ \\
  };
\draw[->]
  (ul) to node[above] {1} (ur)
  ;
  \draw[->]
  (dl) to node[left]{2}(ul)
  ;
  \draw[->]
  (dr) to node[below]{1}(dl)
  ;
  \draw[->]
  (ur) to node[right]{2} (dr)
  ;
    \node [yshift = .4 cm,xshift = -.4 cm] at (ul){k=1};
   \node [yshift = .4 cm,xshift = .4 cm] at (ur){2};
    \node [yshift = -.4 cm,xshift = .4 cm] at (dr){3};
    \node [yshift = -.4 cm,xshift = -.4 cm] at (dl){4};
 
   \matrix[matrix of math nodes,column sep={40pt,between origins},row
    sep={40pt,between origins}] at (4,0)
  {
    & |[name = ul]|\bullet& |[name = ur]|\circ\\
    &|[name=dl]| \circ &|[name=dr]| \circ \\
  };
\draw[->]
  (ur) to node[above] {1} (ul)
  ;
  \draw[->]
  (ul) to node[left]{2}(dl)
  ;
  \draw[->]
  (dr) to node[below]{1}(dl)
  ;
  \draw[->]
  (ur) to node[right]{2} (dr)
  ;
\draw[->]
  (dl) to node[above]{2} (ur)
  ;
    \node [yshift = .4 cm,xshift = -.4 cm] at (ul){k=1};
   \node [yshift = .4 cm,xshift = .4 cm] at (ur){2};
    \node [yshift = -.4 cm,xshift = .4 cm] at (dr){3};
    \node [yshift = -.4 cm,xshift = -.4 cm] at (dl){4};
  \node [yshift = -.95 cm,xshift = .95 cm] at (ul){$C'$};
    \end{tikzpicture}
    \end{enumerate}
	
\caption{Induced subdiagrams of $\Gamma$ and corresponding chordless cycles in $\Gamma' = \mu_k(\Gamma),$ as found in \cite[Figure 5]{BM13}} \label{fig:first_cycle_types}
\end{figure}

\begin{figure}
\begin{enumerate}
    
   \item[(h)]
   \begin{tikzpicture}[baseline=-0.5ex]

    \matrix[matrix of math nodes,column sep={40pt,between origins},row
    sep={40pt,between origins}] at (0,0)
  {
    & |[name = ul]|\circ& |[name = ur]|\bullet\\
    &|[name=dl]| \circ &|[name=dr]| \circ \\
  };
\draw[->]
  (ul) to node[above] {1} (ur)
  ;
  \draw[->]
  (dl) to node[left]{2}(ul)
  ;
  \draw[->]
  (dr) to node[below]{1}(dl)
  ;
  \draw[->]
  (ur) to node[right]{2} (dr)
  ;

    \node [yshift = .4 cm,xshift = -.4 cm] at (ul){1};
   \node [yshift = .4 cm,xshift = .4 cm] at (ur){k=2};
    \node [yshift = -.4 cm,xshift = .4 cm] at (dr){3};
    \node [yshift = -.4 cm,xshift = -.4 cm] at (dl){4};
 
   \matrix[matrix of math nodes,column sep={40pt,between origins},row
    sep={40pt,between origins}] at (4,0)
  {
    & |[name = ul]|\circ & |[name = ur]|\bullet\\
    &|[name=dl]| \circ &|[name=dr]| \circ \\
  };
\draw[->]
  (ur) to node[above] {1} (ul)
  ;
  \draw[->]
  (dl) to node[left]{2}(ul)
  ;
  \draw[->]
  (dr) to node[below]{1}(dl)
  ;
  \draw[->]
  (dr) to node[right]{2} (ur)
  ;
\draw[->]
  (ul) to node[above]{2} (dr)
  ;
    \node [yshift = .4 cm,xshift = -.4 cm] at (ul){1};
   \node [yshift = .4 cm,xshift = .4 cm] at (ur){k=2};
    \node [yshift = -.4 cm,xshift = .4 cm] at (dr){3};
    \node [yshift = -.4 cm,xshift = -.4 cm] at (dl){4};
  \node [yshift = -.95 cm,xshift = .3 5 cm] at (ul){$C'$};
    \end{tikzpicture}
   \item[(i)]
   \begin{tikzpicture}[baseline=-0.5ex]
    \matrix[matrix of math nodes,column sep={40pt,between origins},row
    sep={25pt,between origins}] at (0,0)
  {
    & |[name = ul]|\circ& |[name = um]|\circ & |[name = ur]|\circ&\\
    |[name = ml]|\bullet& &&& |[name = mr]|\vdots \\
    & |[name = dl]|\circ& |[name = dm]|\circ & |[name = dr]|\circ&\\
  };
\draw[->]
  (ml) to node[above] {1} (ul)
  ;
  \draw[->]
  (ul) to node[above]{1}(um)
  ;
  \draw[->]
  (um) to node[above]{1}(ur)
  ;
  \draw[->]
  (ur) [bend left] to (mr)
  ;
  \draw[->]
  (mr) [bend left] to (dr)
  ;
  \draw[->]
  (dr) to node[below]{1}(dm)
  ;
  \draw[->]
  (dm) to node[below]{1}(dl)
  ;
  \draw[->]
  (dl) to node[below]{1}(ml)
  ;
    \node at (ml.west){$k$};
    \node at (ul.north){1};
    \node at (um.north){2};
    \node at (dm.south){$h-1$};
    \node at (dl.south){$h$};
    \matrix[matrix of math nodes,column sep={40pt,between origins},row
    sep={25pt,between origins}] at (8,0)
  {
    & |[name = ul]|\circ& |[name = um]|\circ & |[name = ur]|\circ&\\
    |[name = ml]|\bullet& &&& |[name = mr]|\vdots \\
    & |[name = dl]|\circ& |[name = dm]|\circ & |[name = dr]|\circ&\\
  };
\draw[->]
  (ul) to node[above] {1} (ml)
  ;
  \draw[->]
  (ul) to node[above]{1}(um)
  ;
  \draw[->]
  (um) to node[above]{1}(ur)
  ;
  \draw[->]
  (ur) [bend left] to (mr)
  ;
  \draw[->]
  (mr) [bend left] to (dr)
  ;
  \draw[->]
  (dr) to node[below]{1}(dm)
  ;
  \draw[->]
  (dm) to node[below]{1}(dl)
  ;
  \draw[->]
  (ml) to node[below]{1}(dl)
  ;
  \draw[->]
  (dl) to node[right]{1}(ul)
  ;
    \node at (ml.west){$k$};
    \node at (ul.north){1};
    \node at (um.north){2};
    \node at (dm.south){$h-1$};
    \node at (dl.south){$h$};
    
  \node [yshift = -.9 cm] at (um){$C'$};
    \end{tikzpicture}
   
   \item[(j)]
   \begin{tikzpicture}[baseline=-0.5ex]
    \matrix[matrix of math nodes,column sep={40pt,between origins},row
    sep={25pt,between origins}] at (8,0)
  {
    & |[name = ul]|\circ& |[name = um]|\circ & |[name = ur]|\circ&\\
    |[name = ml]|\bullet& &&& |[name = mr]|\vdots \\
    & |[name = dl]|\circ& |[name = dm]|\circ & |[name = dr]|\circ&\\
  };
\draw[->]
  (ml) to node[above] {1} (ul)
  ;
  \draw[->]
  (ul) to node[above]{1}(um)
  ;
  \draw[->]
  (um) to node[above]{1}(ur)
  ;
  \draw[->]
  (ur) [bend left] to (mr)
  ;
  \draw[->]
  (mr) [bend left] to (dr)
  ;
  \draw[->]
  (dr) to node[below]{1}(dm)
  ;
  \draw[->]
  (dm) to node[below]{1}(dl)
  ;
  \draw[->]
  (dl) to node[below]{1}(ml)
  ;
    \node at (ml.west){$k$};
    \node at (ul.north){1};
    \node at (um.north){2};
    \node at (dm.south){$h-1$};
    \node at (dl.south){$h$};
  \node [yshift = -.9 cm] at (um){$C'$};
    \matrix[matrix of math nodes,column sep={40pt,between origins},row
    sep={25pt,between origins}] at (0,0)
  {
    & |[name = ul]|\circ& |[name = um]|\circ & |[name = ur]|\circ&\\
    |[name = ml]|\bullet& &&& |[name = mr]|\vdots \\
    & |[name = dl]|\circ& |[name = dm]|\circ & |[name = dr]|\circ&\\
  };
\draw[->]
  (ul) to node[above] {1} (ml)
  ;
  \draw[->]
  (ul) to node[above]{1}(um)
  ;
  \draw[->]
  (um) to node[above]{1}(ur)
  ;
  \draw[->]
  (ur) [bend left] to (mr)
  ;
  \draw[->]
  (mr) [bend left] to (dr)
  ;
  \draw[->]
  (dr) to node[below]{1}(dm)
  ;
  \draw[->]
  (dm) to node[below]{1}(dl)
  ;
  \draw[->]
  (ml) to node[below]{1}(dl)
  ;
  \draw[->]
  (dl) to node[right]{1}(ul)
  ;
    \node at (ml.west){$k$};
    \node at (ul.north){1};
    \node at (um.north){2};
    \node at (dm.south){$h-1$};
    \node at (dl.south){$h$};
    \end{tikzpicture}
    
    \item[(k)] $C$ is an oriented cycle in $\Gamma$ not connected to $k$ and $C'$ is the corresponding cycle in $\Gamma'.$
    
      \item[(l)]
   \begin{tikzpicture}[baseline=-0.5ex]
    \matrix[matrix of math nodes,column sep={40pt,between origins},row
    sep={25pt,between origins}] at (8,0)
  {
    & & |[name = um]|\circ & |[name = ur]|\circ&\\
    |[name = ml]|\bullet& |[name = mm]|\circ &&& |[name = mr]|\vdots \\
    & & |[name = dm]|\circ & |[name = dr]|\circ&\\
  };
\draw[->]
  (mm) to node[above] {1} (ml)
  ;
  \draw[->]
  (mm) to node[above]{1}(um)
  ;
  \draw[->]
  (um) to node[above]{1}(ur)
  ;
  \draw[->]
  (ur) [bend left] to (mr)
  ;
  \draw[->]
  (mr) [bend left] to (dr)
  ;
  \draw[->]
  (dr) to node[below]{1}(dm)
  ;
  \draw[->]
  (dm) to node[below]{1}(mm)
  ;
    \node at (ml.west){$k$};
    \node at (mm.north){h};
    \node at (um.north){1};
    \node at (ur.north){2};
    \node at (dr.south){$h-2$};
    \node at (dm.south){$h-1$};
  \node [yshift = -.9 cm] at (ur){$C'$};
   \matrix[matrix of math nodes,column sep={40pt,between origins},row
    sep={25pt,between origins}] at (0,0)
  {
    & & |[name = um]|\circ & |[name = ur]|\circ&\\
    |[name = ml]|\bullet& |[name = mm]|\circ &&& |[name = mr]|\vdots \\
    & & |[name = dm]|\circ & |[name = dr]|\circ&\\
  };
\draw[->]
  (ml) to node[above] {1} (mm)
  ;
  \draw[->]
  (mm) to node[above]{1}(um)
  ;
  \draw[->]
  (um) to node[above]{1}(ur)
  ;
  \draw[->]
  (ur) [bend left] to (mr)
  ;
  \draw[->]
  (mr) [bend left] to (dr)
  ;
  \draw[->]
  (dr) to node[below]{1}(dm)
  ;
  \draw[->]
  (dm) to node[below]{1}(mm)
  ;
    \node at (ml.west){$k$};
    \node at (mm.north){h};
    \node at (um.north){1};
    \node at (ur.north){2};
    \node at (dr.south){$h-2$};
    \node at (dm.south){$h-1$};
  \node [yshift = -.9 cm] at (ur){$C'$};
   \end{tikzpicture}
\end{enumerate}

\caption{Induced subdiagrams of $\Gamma$ and corresponding chordless cycles in $\Gamma' = \mu_k(\Gamma),$ as found in \cite[Figure 6]{BM13}} \label{fig:second_cycle_types}
\end{figure}

\begin{lem}\label{lem:chordless_cycles}
\cite[Lemma 2.5]{BM13} 
Let $\Gamma$ be a diagram of finite type with $\Gamma' = \mu_k(\Gamma),$ the mutation of $\Gamma$ at vertex $k.$ In Figure \ref{fig:first_cycle_types} and Figure \ref{fig:second_cycle_types} we list induced subdiagrams in $\Gamma$ on the left and the resulting induced subdiagrams in $\Gamma'$ with chordless cycles $C'$ on the right, after mutation at $k.$ We draw the diagrams so that $C'$ always has a clockwise cycle. Furthermore, in case $(i),$ we assume $C'$ has at least three vertices, while in case $(j),$ we assume $C'$ has at least four vertices. 

Every chordless cycle in $\Gamma'$ is of one of the types listed in Figure \ref{fig:first_cycle_types} or Figure \ref{fig:second_cycle_types}.
\end{lem}

\section{The Artin Group of a Diagram}
\label{sec:defn_artingroup}

In order to prove our main result, Theorem ~\ref{thm:main}, we must first define the Artin group associated to a finite type diagram. This definition will be similar to that made in \cite{BM13} at the beginning of Section 3, except that we shall not require relation $(R1),$ i.e., $s_i^2 = e.$ Since Artin Groups are very similar to Coxeter groups, with the caveat that the generators are not involutions, we will be able to use these modified relations to great effect.

\subsection{Artin Groups}

\begin{note}
Let
\begin{align*}
\langle x_i,x_j \rangle ^k = \begin{cases}
(x_ix_j)^{\frac{k}{2}}, &\text{ if }k \equiv 0 \pmod 2\\
(x_ix_j)^{\frac{k-1}{2}}x_i &\text{ if } k \equiv 1 \pmod 2
\end{cases}
\end{align*}
That is, $\langle x_i,x_j \rangle$ is just an alternating sequence of $x_i$ and $x_j$ of length $k$.  We also write $\langle x_i,x_j\rangle^{-k}$ to denote $\left(\langle x_i,x_j\rangle^k\right)^{-1}$
\end{note}

\begin{defn}
\cite[Beginning of section 1.2]{C06}
Let $Sym_n(R)$ denote the set of $n \times n$ symmetric matrices with entries in $R$. For $M \in Sym_n(\mathbb Z \cup \infty)$ a symmetric matrix whose entries can take values in the integers or infinity, we define the associated {\it Artin group} in terms of generators and relations by
\begin{align*}
A = \langle x_1,\ldots, x_n| \langle x_i,x_j \rangle^{M_{i,j}} = \langle x_j,x_i \rangle^{M_{i,j}} \text{ for all } i,j \text{ with } M_{i,j}<\infty \rangle,
\end{align*}
\end{defn}

\begin{rem}
Each Artin group has an associated Coxeter group defined by adding in the additional relations $s_i^2 = e$ for all $i.$ An Artin group is said to be of {\it finite type} if its associated Coxeter group is of finite type. To each Artin group of finite type we can assign to it the same Dynkin diagram which is assigned to the Coxeter group associated to the Artin group.
\end{rem}

One of the most well-known Artin groups is the braid group on n strands, which was shown to have an Artin group structure in \cite{FN61}. The associated Coxeter group is the symmetric group $S_{n}$.

\subsection{The Group associated to Diagram}

We are now ready to define the Artin group we associate to a diagram of finite type.

\begin{defn}
\label{defn:cycle_tuple}
Let $(i_0,\ldots, i_{d-1})$ be an ordered tuple such that the subgraph of $\Gamma$ on the vertices $i_0,\ldots, i_{d-1}$ is a chordless cycle, with edges of nonzero weight from $i_k$ to $i_{k+1}$, where subscripts are taken $\pmod d.$ Call such an ordered tuple a {\it chordless cycle tuple.} Then, denote $$p(i_a,i_{a+1}) = s_{i_{a+1}}^{-1}s_{i_{a+2}}^{-1}\dots s_{i_{a-2}}^{-1}s_{i_{a-1}}s_{i_{a-2}}s_{i_{a-3}}\dots s_{i_{a+1}}.$$ Additionally, let $$t(i_a,i_{a+1}) = [s_{i_a},p(i_a,i_{a+1})]$$ where $[a,b] = aba^{-1}b^{-1}$ is the commutator.
\end{defn}

\begin{defn} \label{grp def}
The {\it associated Artin group} to a diagram $\Gamma$ of finite type, denoted $A_\Gamma,$ is generated by $s_i,$ where there is one $s_i$ for each vertex $i$ in $\Gamma.$ These generators are subject to the following relations
\begin{itemize}
\item[(T2)] With $m_{ij}$ as defined in Definition ~\ref{defn:mij}, for all $i \neq j,$ we add the relations
$\langle s_i,s_j \rangle^{m_{ij}}= \langle s_j,s_i \rangle^{m_{ij}}.$

\item[(T3)] Let $(i_0,i_1,\ldots,i_{d-1})$ be a chordless cycle tuple, as defined in Definition ~\ref{defn:cycle_tuple}. If additionally one of the following two conditions hold,
\begin{enumerate}
\item All edges in the chordless cycle are of weight 1 or 2 and the edge $i_{d-1}\rightarrow i_0$ has weight 2,
\item All edges in the chordless cycle have weight $1,$
\end{enumerate}
then, we include the relation
$t(i_0,i_1)=e.$ That is, $s_{i_0}$ and $p(i_0,i_1)$ commute.
\end{itemize}
\end{defn}

\begin{rem}
In the above definition, the chordless cycle tuple is ordered, and so we may have other relations corresponding to chordless cycle tuples which are cyclic reorderings the chordless cycle tuple $(i_0,\ldots, i_{d-1}).$ However, we shall see in Section ~\ref{sec:one_relation} that many of these relations are redundant.
\end{rem}

\begin{rem}
We purposely include relations $(T2),(T3)$ but not $(T1)$ in order to make our relation labeling analogous to that of \cite{BM13} at the beginning of Section 3. Note that if we add the additional relation $(R1)$ as defined at the beginning of Section 3 of \cite{BM13} (namely, if we add $s_i^2 = e$ for all vertices $i$ in $\Gamma$), then we will precisely obtain the group $W_\Gamma$ as defined at the beginning of Section 3 in \cite{BM13}.
\end{rem}

\begin{rem}
Throughout the remainder of the paper, we shall frequently discuss relations on one diagram of finite type, $\Gamma,$ and another diagram of finite type $\Gamma^\prime.$ In order to distinguish the relations in these two groups, we shall refer to the relations on $\Gamma$ as $(T2),(T3)$ and the relations on $\Gamma^\prime$ as $(T2^\prime),(T3^\prime).$
\end{rem}

\begin{ex}
The relations $(T2),(T3)$ in that $\Gamma$ is a square with all edges of weight 1 are as follows:
\[\Gamma =\begin{tikzpicture}[baseline=-0.5ex]
\matrix[matrix of math nodes,column sep={40pt,between origins},row
sep={40pt,between origins}] at (6,0)
{
& |[name = ul]|\circ& |[name = ur]|\circ\\
&|[name=dl]| \circ &|[name=dr]| \circ \\
};
\draw
(ul) -- node[above] {1} (ur)
;
\draw
(dl) -- node[left]{1}(ul)
;
\draw
(dr) -- node[below]{1}(dl)
;
\draw
(ur) -- node[right]{1} (dr)
;
\node [yshift = .4 cm,xshift = -.4 cm] at (ul){1};
\node [yshift = .4 cm,xshift = .4 cm] at (ur){2};
\node [yshift = -.4 cm,xshift = .4 cm] at (dr){3};
\node [yshift = -.4 cm,xshift = -.4 cm] at (dl){4};
\end{tikzpicture}\]
   \begin{enumerate}
\item[(T2)]
\begin{itemize}
\item $\langle s_1,s_2 \rangle^3 = \langle s_2,s_1 \rangle^3,$ i.e., $s_1s_2s_1 = s_2s_1s_2$
\item $\langle s_2,s_3 \rangle^3 = \langle s_3,s_2 \rangle^3,$ i.e., $s_2s_3s_2 = s_3s_2s_3$
\item $\langle s_3,s_4 \rangle^3 = \langle s_4,s_3 \rangle^3,$ i.e., $s_3s_4s_3 = s_4s_3s_4$
\item $\langle s_4,s_1 \rangle^3 = \langle s_1,s_4 \rangle^3,$ i.e., $s_4s_1s_4 = s_1s_4s_1$
\end{itemize}
\item[(T3)]
\begin{itemize}
\item $s_{1}s_{2}^{-1}s_{3}^{-1}s_4s_{3}s_2s_1^{-1}s_2^{-1}s_3^{-1}s_4^{-1}s_3s_2 = e$
\item $s_{2}s_{3}^{-1}s_{4}^{-1}s_1s_{4}s_3s_2^{-1}s_3^{-1}s_4^{-1}s_1^{-1}s_4s_3 = e$
\item $s_{3}s_{4}^{-1}s_{1}^{-1}s_2s_{1}s_4s_3^{-1}s_4^{-1}s_1^{-1}s_2^{-1}s_1s_4 = e$
\item $s_{4}s_{1}^{-1}s_{2}^{-1}s_3s_{2}s_1s_4^{-1}s_1^{-1}s_2^{-1}s_3^{-1}s_2s_1 = e$
\end{itemize}
\end{enumerate}
\end{ex}

\begin{ex}
The relations $(T2),(T3)$ in that $\Gamma$ is a triangle with two edges of weight 2 and one of weight 1 are as follows:
\[\Gamma = \begin{tikzpicture}[baseline=-0.5ex]
\matrix[matrix of math nodes,column sep={25pt,between origins},row
    sep={40pt,between origins}] at (0,0)
  {
    && |[name = u]|\circ& \\
    &|[name=dl]| \circ & &|[name=dr]| \circ \\
  };
\draw[->]
   (u) to node[right] {2}(dr)
  ;
  \draw[->]
  (dl) to node[left]{2}(u)
  ;
  \draw[->]
  (dr) to node[below]{1}(dl)
  ;
  \node [yshift = .4 cm] at (u){1};
    \node [yshift = -.4 cm,xshift = .4 cm] at (dr){2};
    \node [yshift = -.4 cm,xshift = -.4 cm] at (dl){3};
   \end{tikzpicture}\]
   \begin{enumerate}
\item[(T2)]
\begin{itemize}
\item $\langle s_1,s_2 \rangle^4 = \langle s_2,s_1 \rangle^4,$ i.e., $s_1s_2s_1s_2 = s_2s_1s_2s_1$
\item $\langle s_2,s_3 \rangle^3 = \langle s_3,s_2 \rangle^3,$ i.e., $s_2s_3s_2 = s_3s_2s_3$
\item $\langle s_3,s_4 \rangle^4 = \langle s_4,s_3 \rangle^4,$ i.e., $s_3s_4s_3s_4 = s_4s_3s_4s_3$              
\end{itemize}
\item[(T3)]
\begin{itemize}
\item $s_1s_2^{-1}s_3^{-1}s_2s_1^{-1}s_2^{-1}s_3s_2 = e$

\item $s_2s_3^{-1}s_1^{-1}s_3s_2^{-1}s_3^{-1}s_1s_3 =e$

\end{itemize}
\end{enumerate}
\end{ex}

\begin{rem}
Note that if $\Gamma$ is the graph associated to a Dynkin diagram, then $W_\Gamma$ as we have defined it is precisely the Artin group corresponding to that Dynkin diagram. This occurs because, in this case, we have no cycles in $\Gamma,$ and so we only have relations of the form $(T2),$ which define the Artin group.
\end{rem}

\section{Symmetry among the (R3) Relations}
\label{sec:one_relation}

Given the relations (T2), many of the relations in (T3) become redundant. For example,

\begin{lem} \label{Sym-Lem}
Let $\Gamma$ be a diagram of finite type which contains a chordless cycle C:
\begin{equation*}
\begin{tikzcd}
i_0 \arrow{r} & i_1 \arrow{r} & \cdots \arrow{r} & i_{d-1} \arrow{r} & i_0
\end{tikzcd}
\end{equation*}
so that all edges have weight 1. Then if W is a group generated by $s_{1}, \dots, s_{n}$ satisfying the relations (T2) and $t(i_{a}, i_{a+1}) = e$ for some a $\in \{1, \dots, d\}$, all of the relations in (T3) hold for C.
\end{lem}

In the proof of this Lemma and throughout the rest of the paper, we will frequently employ the following relations, which follow from the (T2) relations.

\begin{lem} \label{Extra-Rel}
Let $\Gamma$ be a diagram of finite type, and let $A_{\Gamma}$ be generated by $s_{1}, \dots, s_{n}$. Then we have that 
\begin{enumerate}[(a)]
\item $s_{i}s_{j}s_{i}^{-1} = s_{j}^{-1}s_{i}s_{j}$ if there is an arrow of weight 1 from i to j in $\Gamma$ \\
\item $s_{i}s_{j}s_{i}^{-1}s_{j}^{-1} = s_{j}^{-1}s_{i}^{-1}s_{j}s_{i}$ if there is an arrow of weight 2 from i to j in $\Gamma$
\end{enumerate}
\end{lem}

\begin{proof}[Proof of \ref{Extra-Rel}]
The relation in (a) follows immediately from the (T2) relation $s_{j}s_{i}s_{j} = s_{i}s_{j}s_{i}$ by left multiplying both sides by $s_{j}^{-1}$ and right multiplying both sides by $s_{i}^{-1}$. Similarly, the relation in (b) arises from the (T2) relation $s_{j}s_{i}s_{j}s_{i} = s_{i}s_{j}s_{i}s_{j}$ by left multiplying both sides by $s_{i}^{-1}s_{j}^{-1}$ and right multiplying both sides by $s_{j}^{-1}s_{i}^{-1}$
\end{proof}   
\begin{proof}[Proof of \ref{Sym-Lem}]
It suffices to prove that the relation t(0, 1) = e implies that t(d-1, 0) = e, as the other relations will follow by induction. So suppose $A_{\Gamma}$ satisfies the relation t(0, 1) = e. Then we have
\begin{align*}
s_{d-1}p(d-1, 0) & = s_{d-1}s_{0}^{-1}s_{1}^{-1}\dots s_{d-3}^{-1}s_{d-2}s_{d-3}\dots s_1s_0 \\
&= s_{0}^{-1}s_{0}s_{d-1}s_{0}^{-1}s_{1}^{-1}\dots s_{d-3}^{-1}s_{d-2}s_{d-3}\dots s_{1}s_{d-1}^{-1}s_{d-1}s_{0} \\
&= s_{0}^{-1}s_{d-1}^{-1}s_{0}s_{d-1}s_{1}^{-1}\dots s_{d-3}^{-1}s_{d-2}s_{d-3}\dots s_{1}s_{d-1}^{-1}s_{d-1}s_{0} &\text{by (T2)} \\
&= s_{0}^{-1}s_{d-1}^{-1}s_{0}s_{1}^{-1}\dots s_{d-3}^{-1}s_{d-1}s_{d-2}s_{d-1}^{-1}s_{d-3}\dots s_{1}s_{d-1}s_{0} \\
&= s_{0}^{-1}s_{d-1}^{-1}(s_{0}s_{1}^{-1}\dots s_{d-3}^{-1}s_{d-2}^{-1}s_{d-1}s_{d-2}s_{d-3}\dots s_{1})s_{d-1}s_{0} &\text{by (T2)}\\
&= s_{0}^{-1}s_{d-1}^{-1}(s_{1}^{-1} \dots s_{d-3}^{-1}s_{d-2}^{-1}s_{d-1}s_{d-2}s_{d-3}\dots s_{0})s_{d-1}s_{0} &\text{by t(0, 1) = e}\\
&= s_{0}^{-1}s_{d-1}^{-1}(s_{1}^{-1} \dots s_{d-3}^{-1}s_{d-1}s_{d-2}s_{d-1}^{-1}s_{d-3}\dots s_{0})s_{d-1}s_{0} &\text{by (T2)}\\
&= s_{0}^{-1}(s_{d-1}^{-1}s_{d-1})s_{1}^{-1}\dots s_{d-3}^{-1}s_{d-2}s_{d-3}\dots s_{1}s_{d-1}^{-1}s_{0}s_{d-1}s_{0} \\
&= s_{0}^{-1}s_{1}^{-1}\dots s_{d-3}^{-1}s_{d-2}s_{d-3}\dots s_{1}s_{d-1}^{-1}s_{0}s_{d-1}s_{0} \\
&= s_{0}^{-1}s_{1}^{-1}\dots s_{d-3}^{-1}s_{d-2}s_{d-3}\dots s_{1}s_{0}s_{d-1}s_{0}^{-1}s_{0} &\text{by (T2)}\\
&= s_{0}^{-1}s_{1}^{-1}\dots s_{d-3}^{-1}s_{d-2}s_{d-3}\dots s_{1}s_{0}s_{d-1} = p(d-1, 0)s_{d-1}
\end{align*}
as required. Note that line 3 is equal to 4 and line 7 is equal to line 8 since the cycle is chordless, meaning that $s_{d-1}$ commutes with every element except $s_{0}$ and $s_{d-2}$.
\end{proof}
Furthermore, we obtain similar results for cycles containing edges of weight 2.
\begin{lem} \label{2tri}
Let $\Gamma$ be a diagram of finite type containing the following 3-cycle:

\begin{tikzpicture}[baseline=-0.5ex]
   \matrix[matrix of math nodes,column sep={25pt,between origins},row
    sep={40pt,between origins}] at (10,0)
  {
    && |[name = u]| {k} \circ& \\
    &|[name=dl]|{j} \circ & &|[name=dr]| \circ {i} \\
  };
\draw[->]
  (u) to node[right] {2} (dr)
  ;
  \draw[->]
  (dl) to node[left]{2}(u)
  ;
  \draw[->]
  (dr) to node[below]{1}(dl)
  ;
  
   \end{tikzpicture}
and let A be the group with generators $s_{1}, \dots, s_{n}$ defined by $\Gamma$. Then the relations t(i, j) = e and t(k, i) = e are equivalent.
\end{lem}

\begin{proof}
The lemma follows from the fact that
\begin{align*}
& s_{k}^{-1}s_{j}(s_{i}p(i, j)s_{i}^{-1}p(i, j)^{-1})s_{j}^{-1}s_{k} \\
&= s_{k}^{-1}s_{j}(s_{i}s_{j}^{-1}s_{k}s_{j}s_{i}^{-1}s_{j}^{-1}s_{k}^{-1}s_{j})s_{j}^{-1}s_{k} \\
&= s_{k}^{-1}s_{j}s_{i}s_{j}^{-1}s_{k}s_{j}s_{i}^{-1}s_{j}^{-1} \\
&= s_{k}^{-1}s_{i}^{-1}s_{j}s_{i}s_{k}s_{i}^{-1}s_{j}^{-1}s_{i} \\
&= s_{k}^{-1}p(k, i)s_{k}p(k, i)^{-1}
\end{align*}
\end{proof}

In the setting of the previous lemma, we also obtain the following relation, which will play an important role in later proofs.

\begin{lem}\label{lem:extra_square_relation}
Suppose $\Gamma$ contains a 3-cycle with edges of weight 2, labeled as in ~\ref{2tri}, and suppose that A$_{\Gamma}$ is generated by $s_{1}, \dots s_{n}$. Then we have that
$$s_{j}p(j, k)s_{j}p(j, k)^{-1}s_{j}^{-1}p(j, k)^{-1} := s_{j}s_{k}^{-1}s_{i}s_{k}s_{j}s_{k}^{-1}s_{i}^{-1}s_{k}s_{j}^{-1}s_{k}^{-1}s_{i}^{-1}s_{k} = e.$$
\end{lem}

\begin{proof}
We show that
$$s_{j}^{-1}s_{k}^{-1}s_{i}s_{k}s_{j}s_{k}^{-1}s_{i}s_{k}s_{j}^{-1}s_{k}^{-1}s_{i}^{-1}s_{k} = e.$$
The result then follows by inverting the relation and conjugating by $s_{j}$. In the following computation, we will underline the terms being manipulated in each line for emphasis.

\begin{align*}
& s_{k}s_{j}s_{k}(s_{j}^{-1}s_{k}^{-1}s_{i}s_{k}s_{j}s_{k}^{-1}s_{i}s_{k}s_{j}^{-1}s_{k}^{-1}s_{i}^{-1}s_{k})s_{k}^{-1}s_{j}^{-1}s_{k}^{-1} \\
&= s_{k}\underline{s_{j}s_{k}s_{j}^{-1}s_{k}^{-1}}s_{i}s_{k}s_{j}s_{k}^{-1}s_{i}s_{k}s_{j}^{-1}s_{k}^{-1}s_{i}^{-1}\underline{s_{k}s_{k}^{-1}}s_{j}^{-1}s_{k}^{-1} \\
&= \underline{s_{k}s_{k}^{-1}}\underline{s_{j}^{-1}s_{k}s_{j}s_{i}}s_{k}s_{j}s_{k}^{-1}s_{i}s_{k}s_{j}^{-1}s_{k}^{-1}s_{i}^{-1}s_{j}^{-1}s_{k}^{-1} \\
&= s_{i}\underline{s_{j}^{-1}s_{k}s_{j}s_{k}}s_{j}s_{k}^{-1}s_{i}s_{k}s_{j}^{-1}s_{k}^{-1}s_{i}^{-1}s_{j}^{-1}s_{k}^{-1} \\
&= s_{i}s_{k}s_{j}\underline{s_{k}s_{j}^{-1}s_{j}s_{k}^{-1}}s_{i}s_{k}s_{j}^{-1}s_{k}^{-1}s_{i}^{-1}s_{j}^{-1}s_{k}^{-1} \\
&= s_{i}\underline{e}s_{k}s_{j}s_{i}s_{k}s_{j}^{-1}s_{k}^{-1}s_{i}^{-1}s_{j}^{-1}s_{k}^{-1} \\
&= s_{i}s_{j}\underline{s_{j}^{-1}s_{k}s_{j}s_{i}}s_{k}s_{j}^{-1}s_{k}^{-1}s_{i}^{-1}s_{j}^{-1}s_{k}^{-1} \\
&= \underline{s_{i}s_{j}s_{i}}s_{j}^{-1}s_{k}s_{j}s_{k}s_{j}^{-1}s_{k}^{-1}s_{i}^{-1}s_{j}^{-1}s_{k}^{-1} \\
&= s_{j}s_{i}\underline{s_{j}s_{j}^{-1}}s_{k}\underline{s_{j}s_{k}s_{j}^{-1}s_{k}^{-1}}s_{i}^{-1}s_{j}^{-1}s_{k}^{-1} \\
&= s_{j}s_{i}\underline{s_{k}s_{k}^{-1}}\underline{s_{j}^{-1}s_{k}s_{j}s_{i}^{-1}}s_{j}^{-1}s_{k}^{-1} \\
&= s_{j}\underline{s_{i}s_{i}^{-1}}s_{j}^{-1}s_{k}\underline{s_{j}s_{j}^{-1}}s_{k}^{-1} \\
&= \underline{s_{j}s_{j}^{-1}}\underline{s_{k}s_{k}^{-1}} \\
&= e
\end{align*}
\end{proof}

\begin{lem} \label{2sqr}
Let $\Gamma$ be a diagram of finite type containing the following 4-cycle:
\begin{tikzpicture}[baseline=-0.5ex]
\matrix[matrix of math nodes,column sep={40pt,between origins},row
    sep={40pt,between origins}] at (6,0)
  {
    & |[name = ul]|{i}\circ& |[name = ur]|\circ{j}\\
    &|[name=dl]|{l} \circ &|[name=dr]| \circ{k} \\
  };
\draw[->]
  (ul) to node[above] {1} (ur)
  ;
  \draw[->]
  (dl) to node[left]{2}(ul)
  ;
  \draw[->]
  (dr) to node[below]{1}(dl)
  ;
  \draw[->]
  (ur) to node[right]{2} (dr)
  ;
\end{tikzpicture}
and let A be the group with generators $s_{1}, \dots, s_{n}$ defined by $\Gamma$. Then the relations t(i, j) = e and t(k, l) = e are equivalent.
\end{lem}

\begin{proof}
We have that
\begin{align*}
& s_{k}^{-1}s_{l}^{-1}s_{j}(s_{i}p(i, j)s_{i}^{-1}p(i, j)^{-1})s_{j}^{-1}s_{l}s_{k} \\
&= s_{k}^{-1}s_{l}^{-1}s_{j}(s_{i}s_{j}^{-1}s_{k}^{-1}s_{l}s_{k}s_{j}s_{i}^{-1}s_{j}^{-1}s_{k}^{-1}s_{l}^{-1}s_{k}s_{j})s_{j}^{-1}s_{l}s_{k} \\
&= s_{k}^{-1}s_{l}^{-1}(s_{j}s_{i}s_{j}^{-1})(s_{k}^{-1}s_{l}s_{k})(s_{j}s_{i}^{-1}s_{j}^{-1})s_{k}^{-1}s_{l}^{-1}s_{k}s_{l}s_{k} \\
&= s_{k}^{-1}s_{l}^{-1}s_{i}^{-1}s_{j}s_{i}s_{l}s_{k}s_{l}^{-1}s_{i}^{-1}s_{j}^{-1}s_{i}s_{k}^{-1}(s_{l}^{-1}s_{k}s_{l})s_{k} \\
&= s_{k}^{-1}s_{l}^{-1}s_{i}^{-1}s_{j}s_{i}s_{l}s_{k}s_{l}^{-1}s_{i}^{-1}s_{j}^{-1}s_{i}s_{l} \\
&= s_{k}^{-1}p(k, l)s_{k}p(k, l)^{-1}
\end{align*}
\end{proof}

Finally, we conclude the section by establishing a relationship between the groups defined by $\Gamma$ and $\Gamma^{op}$, the diagram obtained by reversing all arrows in $\Gamma$.

\begin{lem}\label{lem:op_homomorphism}
Let A$_{\Gamma}$ be generated by $s_{1}, \dots, s_{n}$, and let A$_{\Gamma^{op}}$ be generated by $r_{1}, \dots, r_{n}$. Then the map $$\Delta: s_{i} \rightarrow r_{i}^{-1}$$ defines an isomorphism between A$_{\Gamma}$ and A$_{\Gamma^{op}}.$
\end{lem}
\begin{proof}
It suffices to show that the map is well-defined, or that the elements $s_{1}^{-1}, \dots, s_{n}^{-1}$ satisfy the relations (T2) and (T3) of A$_{\Gamma^{op}}$. One can see that the inverse elements satisfy (T2) in A$_{\Gamma^{op}}$ by taking the inverse of both sides of the corresponding relation in A$_{\Gamma}$. To see that the elements satisfy (T3) in A$_{\Gamma^{op}}$, note that for a chordless cycle in $\Gamma$ with all weights equal to one, we have $$s_{0}s_{1}^{-1}\dots s_{d-2}^{-1}s_{d-1}s_{d-2}\dots s_{1} = s_{1}^{-1}\dots s_{d-2}^{-1}s_{d-1}s_{d-2}\dots s_{1}s_{0}$$ by the relation t(0,1) = e in (T3) in A$_{\Gamma}$. But then applying relations from (T2), we have that $$s_{0}s_{1}^{-1}\dots s_{d-3}^{-1}s_{d-1}s_{d-2}s_{d-1}^{-1}s_{d-3}\dots s_{1} = s_{1}^{-1}\dots s_{d-1}s_{d-2}s_{d-1}^{-1}\dots s_{1}s_{0},$$ and since the cycle is chordless, we then have $$s_{0}s_{d-1}s_{1}^{-1}\dots s_{d-3}^{-1}s_{d-2}s_{d-3}\dots s_{1}s_{d-1}^{-1} = s_{d-1}s_{1}^{-1}\dots s_{d-3}^{-1}s_{d-2}s_{d-3}\dots s_{1}s_{d-1}^{-1}s_{0}.$$ Repeating this process, we find that $$s_{0}s_{d-1}s_{d-2}\dots s_{2}s_{1}s_{2}^{-1}\dots s_{d-2}^{-1}s_{d-1}^{-1} = s_{d-1}s_{d-2}\dots s_{2}s_{1}s_{2}^{-1}\dots s_{d-2}^{-1}s_{d-1}^{-1}s_{0}.$$ But this occurs if and only if $s_{1}^{-1}, \dots, s_{n}^{-1}$ satisfies the relation t(0, d-1) = e in A$_{\Gamma^{op}}.$

For a triangle labeled as in ~\ref{2tri}, by the relation t(k, i) = e we have $$s_{k}s_{i}^{-1}s_{j}s_{i}s_{k}^{-1} = s_{i}^{-1}s_{j}s_{i}.$$ Hence $$s_{k}s_{j}s_{i}s_{j}^{-1}s_{k}^{-1} = s_{j}s_{i}s_{j}^{-1}.$$ But as before, this can occur if and only if $s_{i}^{-1}, s_{j}^{-1}, s_{k}^{-1}$ satisfy the relation t(k, j) = e in A$_{\Gamma^{op}}$.

Finally, given a square labeled as in ~\ref{2sqr} and the relations t(1, 2) = e and t(3, 4) = e, we have
\begin{align*}
& s_{j}s_{i}s_{l}s_{k}s_{l}^{-1}s_{i}^{-1} \\
&= s_{i}s_{i}^{-1}s_{j}s_{i}s_{l}s_{k}s_{l}^{-1}s_{i}^{-1} \\
&= s_{i}s_{j}s_{i}s_{j}^{-1}s_{k}^{-1}s_{l}s_{k}s_{i}^{-1} \\
&= s_{i}s_{j}(s_{i}s_{j}^{-1}s_{k}^{-1}s_{l}s_{k}s_{j})s_{j}^{-1}s_{i}^{-1} \\
&= s_{i}s_{j}s_{j}^{-1}(s_{k}^{-1}s_{l}s_{k})s_{j}(s_{i}s_{j}^{-1}s_{i}^{-1}) \\
&= s_{i}s_{l}s_{k}s_{l}^{-1}s_{j}s_{j}^{-1}s_{i}^{-1}s_{j} \\
&= s_{i}s_{l}s_{k}s_{l}^{-1}s_{i}^{-1}s_{j}
\end{align*}
But this relation holds if and only if $s_{i}^{-1}, \dots, s_{l}^{-1}$ satisfy t(j, i) = e in A$_{\Gamma^{op}}$. Therefore, we are done.
\end{proof}

\section{Main Result}\label{sec:main_result}
\noindent In Section \ref{sec:proof_of_main_result} we prove our main result:

\begin{thm}\label{thm:main}
Let $\Gamma$ be a diagram of finite type, and let $\Gamma^{\prime} = \mu_k(\Gamma)$ be the mutation of $\Gamma$ at vertex $k$. Then $A_{\Gamma} \cong A_{\Gamma^{\prime}}$
\end{thm}

The structure of the proof will be analogous to the structure of the proof of Theorem A in \cite{BM13}, but with most of the details changed to account for the fact that the $(R1)$ relations are not included.  In particular, many of the computations used to prove the supporting lemmas and propositions of Theorem A rely heavily on the $(R1)$ relations and therefore do not apply to the supporting lemmas and propositions of Theorem \ref{thm:main}.

Throughout the section we will fix a diagram of finite type $\Gamma$, a vertex $k$ of $\Gamma$, and write $\Gamma^{\prime} = \mu_k(\Gamma)$.  We will write $s_i$, $r_i$, $q_i$, and $u_i$ for the generators corresponding to vertex $i$ of $A_{\Gamma}$, $A_{\Gamma^\prime}$, $A_{\Gamma^{op}}$, and $A_{\left(\Gamma^\prime\right)^{op}}$, respectively. Note that the $u_i$ are generators $A_{\left(\Gamma^{op}\right)^\prime}$ as well, since $\left(\Gamma^\prime\right)^{op} = \left(\Gamma^{op}\right)^\prime$. In the proof of Theorem \ref{thm:main} we will use Lemma \ref{lem:op_homomorphism} along with the following proposition, which we prove in Section \ref{sec:proof_of_ti_relations}.

\begin{prop}\label{prop:ti_relations}
The map $\varphi \colon A_{\Gamma^\prime}\rightarrow A_{\Gamma}$ defined by
\begin{displaymath}
\varphi(r_i) = \begin{cases} s_ks_is_k^{-1} & \mbox{if there is a (possibly weighted) arrow } i \rightarrow k \mbox{ in } \Gamma\\
s_i & \mbox{otherwise}\\
\end{cases}
\end{displaymath}

\noindent is a group homomorphism.
\end{prop}

\subsection{Proof of Main Result}\label{sec:proof_of_main_result}

\begin{proof}[Proof of Theorem \ref{thm:main}]
By Proposition \ref{prop:ti_relations} $\varphi\colon A_{\Gamma^\prime}\rightarrow A_\Gamma$ is a group homomorphism and $\varphi_{op}\colon A_{\Gamma^{op}}\rightarrow A_{\left(\Gamma^{op}\right)^\prime}$ defined by

\begin{displaymath}
\varphi_{op}(q_i) = \begin{cases}    u_ku_iu_k^{-1} & \mbox{if there is a (possibly weighted) arrow } i \rightarrow k \mbox{ in } \left(\Gamma^{op}\right)^\prime\\
				u_i & \mbox{otherwise}\\
	\end{cases}
\end{displaymath}

\noindent is a group homomorphism as well.  By Lemma \ref{lem:op_homomorphism}, there exist two well-defined homomorphisms $\Delta\colon A_{\Gamma}\rightarrow A_\Gamma^{op}$ defined by $\Delta(s_i) = q_i^{-1}$ and $\Delta^\prime\colon A_{\left(\Gamma^\prime\right)^{op}} \rightarrow A_{\Gamma^\prime}$ defined by $\Delta^\prime(u_i) = r_i^{-1}$. We then have a homomorphism 

$$\psi = \Delta^\prime\circ\varphi_{op}\circ\Delta\colon A(\Gamma)\rightarrow A(\Gamma^{op})\rightarrow A(\left(\Gamma^{op}\right)^{\prime}) \rightarrow A(\Gamma^{\prime})$$

Suppose that there is an arrow $i\rightarrow k$ in $\Gamma$. Then there will be an arrow $k\rightarrow i$ in $\Gamma^{op}$ and hence an arrow $i\rightarrow k$ in $\left(\Gamma^{op}\right)^{\prime}$, so we have that

$$\psi\circ\varphi(r_i) = \Delta^\prime(\varphi_{op}(\Delta(\varphi(r_i)))) = \Delta^\prime(\varphi_{op}(\Delta(s_ks_is_k^{-1}))) = \Delta^\prime(\varphi_{op}(q_k^{-1}q_i^{-1}q_k)) = \Delta^\prime(u_i^{-1}) = r_i$$

Similarly if there is an arrow $k\rightarrow i$ or no arrow between $i$ and $k$ in $\Gamma$ then there will be an arrow $k\rightarrow i$ or no arrow between $i$ and $k$ in $\left(\Gamma^{op}\right)^{\prime}$, respectively. In each of these cases we have that

$$\psi\circ\varphi(r_i) = \Delta^\prime(\varphi_{op}(\Delta(\varphi(r_i)))) = \Delta^\prime(\varphi_{op}(\Delta(s_i))) = \Delta^\prime(\varphi_{op}(q_i^{-1})) = \Delta^\prime(u_i^{-1}) = r_i$$

In other words, starting at any node in the square below and following the maps around gives the identity map.

\begin{figure*}[h]
\hspace*{-4em}
\begin{tikzpicture}[baseline=-0.5ex]
\matrix[matrix of math nodes,column sep={60pt,between origins},row
    sep={60pt,between origins}] at (6,0)
  {
    & |[name = ul]|{A_{\Gamma^{op}}}\circ& |[name = ur]|\circ{A_{\left(\Gamma^\prime\right)^{op}}}\\
    &|[name=dl]|{A_{\Gamma}} \circ &|[name=dr]| \circ{A_{\Gamma^\prime}} \\
  };
\draw[->]
  (ul) to node[above]{$\varphi_{op}$} (ur)
  ;
  \draw[->]
  (dl) to node[left]{$\Delta$}(ul)
  ;
  \draw[->]
  (dr) to node[below]{$\varphi$}(dl)
  ;
  \draw[->]
  (ur) to node[right]{$\Gamma^\prime$} (dr)
  ;
\end{tikzpicture}
\end{figure*}

Thus $\psi\circ\varphi$ is the identity map on $A_{\Gamma^{\prime}}$.  By a similar argument $\varphi\circ\psi$ is the identity map on $A_\Gamma$, and hence $A_\Gamma\cong A_{\Gamma^{\prime}}$.
\end{proof}

\subsection{Proof of Proposition \ref{prop:ti_relations}}\label{sec:proof_of_ti_relations}

We prove Proposition \ref{prop:ti_relations} by showing that that the elements $\varphi(r_i)\in A_\Gamma$ satisfy the $(T2^\prime)$ and $(T3^\prime)$ relations in $A_{\Gamma^\prime}$.  The proof that the $\varphi(r_i)$ satisfy these relations is divided among Lemmas \ref{lem:r2a_relations}, \ref{lem:r2b_relations}, and \ref{lem:r3_relations}.  Throughout the proofs we write $t_i = \varphi(r_i)$ and $m_{ij}^\prime$ for the weight of the edge between $i$ and $j$ in $\Gamma^\prime$.

\begin{lem}\label{lem:r2a_relations}
Let $i,j$ be distinct vertices of $\Gamma$.
\begin{enumerate}[(a)]
\item If $i=k$ or $j=k$, then $\langle t_it_j \rangle^{m_{ij}^\prime} = \langle t_jt_i \rangle^{m_{ij}^\prime}$.
\item If at most one of $i,j$ is connected to $k$ in $\Gamma$, then $\langle t_it_j \rangle^{m_{ij}^\prime} = \langle t_jt_i \rangle^{m_{ij}^\prime}$.
\end{enumerate}
\end{lem}

\begin{proof}
For case $(a)$, suppose without loss of generality that $i=k$.  Note that $m_{ij}^\prime = m_{ij}$. The only nontrivial case is when there is an arrow $j\rightarrow k = i$.  Since $i$ and $j$ are connected in this case, $m_{ij}$ is one of 3, 4, or 6.

\vskip.1in
\noindent
{\sf Case $m_{ij} = 3$}.
Here $\langle s_js_i \rangle^{3} = \langle s_is_j \rangle^{3}$, so $s_is_j = s_js_is_js_i^{-1}$ and we have

$$\langle t_it_j \rangle^{3} = t_it_jt_i = s_is_is_js_i^{-1}s_i = s_is_is_j = s_is_js_is_js_i^{-1} = t_jt_it_j = \langle t_jt_i \rangle^{3}$$

\vskip.1in
\noindent
{\sf Case $m_{ij} = 4$}.  Here $\langle s_is_j \rangle^{4} = \langle s_js_i \rangle^{4}$, so $s_is_is_js_is_js_i^{-1} = s_is_js_is_j$ and therefore

$$\langle t_it_j \rangle^{4} = s_is_is_js_i^{-1}s_is_is_js_i^{-1} = s_is_is_js_is_js_i^{-1} = s_is_js_is_j = s_is_js_i^{-1}s_is_is_js_i^{-1}s_i = \langle t_jt_i \rangle^{4}$$

\vskip.1in
\noindent
{\sf Case $m_{ij} = 6$}.  Here $\langle s_is_j \rangle^{6} = \langle s_js_i \rangle^{6}$, so $s_is_is_js_is_js_is_js_i^{-1} = s_is_js_is_js_is_j$.  As in the previous case, we add and remove pairs $s_is_i^{-1}$ as necessary, giving

$$\langle t_it_j \rangle^{6} = s_is_is_js_i^{-1}s_is_is_js_i^{-1}s_is_is_js_i^{-1} = s_is_js_i^{-1}s_is_is_js_i^{-1}s_is_is_js_i^{-1}s_i = \langle t_it_j \rangle^{6}$$

For case $(b)$, the only nontrivial case is when there is an arrow $i\rightarrow k$ or $j\rightarrow k$. Without loss of generality, suppose there is an arrow $i\rightarrow k$. Since $j$ is not connected to $k$, we know that $s_js_k = s_ks_j$.

\vskip.1in
\noindent
{\sf Case $m_{ij} = 2$}.  Here $s_is_j = s_js_i$, so $s_j$ commutes with both $s_i$ and $s_k$ and we have that

$$t_it_j = s_ks_is_k^{-1}s_j = s_js_ks_is_k^{-1} = t_jt_i$$

\vskip.1in
\noindent
{\sf Case $m_{ij} = 3$}.  Here $\langle s_is_j \rangle^{3} = \langle s_js_i \rangle^{3}$, so we have that 

$$\langle t_it_j \rangle^{3} = s_ks_is_k^{-1}s_js_ks_is_k^{-1} = s_ks_is_js_is_k^{-1} = s_ks_js_is_js_k^{-1} = s_js_ks_is_k^{-1}s_j = \langle t_jt_i \rangle^{3}$$

\vskip.1in
\noindent
{\sf Case $m_{ij} = 4$}.  Here $\langle s_is_j \rangle^{4} = \langle s_js_i \rangle^{4}$, so we have that

$$\langle t_it_j \rangle^{4} = s_ks_is_k^{-1}s_js_ks_is_k^{-1}s_j = s_ks_is_js_is_js_k^{-1} = s_ks_js_is_js_is_k^{-1} = s_js_ks_is_k^{-1}s_js_ks_is_k^{-1} = \langle t_jt_i \rangle^{4}$$

\vskip.1in
\noindent
{\sf Case $m_{ij} = 6$}.  Here $\langle s_is_j \rangle^{6} = \langle s_js_i \rangle^{6}$, so we have that

$$\langle t_it_j \rangle^{6} = s_ks_is_k^{-1}s_js_ks_is_k^{-1}s_js_ks_is_k^{-1}s_j = s_ks_is_js_is_js_is_js_k^{-1} = s_ks_js_is_js_is_js_is_k^{-1} = \langle t_it_j \rangle^{6}$$
\end{proof}

\begin{lem}\label{lem:r2b_relations}
Let $i,j$ be distinct vertices of $\Gamma$ such that $i$ and $j$ are connected.  Then $\langle t_it_j \rangle^{m_{ij}^\prime} = \langle t_jt_i \rangle^{m_{ij}^\prime}$.
\end{lem}

\begin{proof}
The possibilities for the subdiagram induced by $i$, $j$, and $k$ are enumerated in Figure \ref{fig:local_three}.  We show that $t_i$ and $t_j$ satisfy the $(T2^\prime)$ relations by checking each case.  Within each case, subcase $(i)$ is when the subdiagram of $\Gamma$ is the diagram on the left in Figure \ref{fig:local_three}, and subcase $(ii)$ is when the subdiagram of $\Gamma$ is the diagram on the right in Figure \ref{fig:local_three}.

Throughout the proof we will make frequent use of the fact that if $m$ and $n$ are vertices of $\Gamma$, then

$$s_ms_ns_m = s_ns_ms_n \Leftrightarrow s_ms_ns_m^{-1} = s_n^{-1}s_ms_n \Leftrightarrow s_ms_n^{-1}s_m^{-1} = s_n^{-1}s_m^{-1}s_n \Leftrightarrow s_m^{-1}s_n^{-1}s_m^{-1} = s_n^{-1}s_m^{-1}s_n^{-1}$$

When helpful, we underline the sections of an expression that are about to be manipulated. We also frequently combine two applications of $A_{\Gamma}$ relations when one manipulation is simply commuting pairs of variables.

\begin{enumerate}[a)]
\item
\begin{enumerate}[i)]
\item We have $\langle t_it_j \rangle^{2} = s_ks_is_k^{-1}s_ks_js_k^{-1} = s_ks_is_js_k^{-1} = s_ks_js_is_k^{-1} = \langle t_jt_i \rangle^{2}$.
\item We have $\langle t_it_j \rangle^{2} = s_is_j = s_js_i = \langle t_jt_i \rangle^{2}$.
\end{enumerate}
\item
\begin{enumerate}[i)]
\item We have
\begin{align*}
\langle t_it_j \rangle^{3} &= s_ks_i\underline{s_k^{-1}s_js_k}s_is_k^{-1}\\
&= s_k\underline{s_is_j}s_k\underline{s_j^{-1}s_i}s_k^{-1}\\
&= s_ks_j\underline{s_is_ks_i}s_j^{-1}s_k^{-1}\\
&= \underline{s_ks_js_k}s_i\underline{s_ks_j^{-1}s_k^{-1}}\\
&= s_js_k\underline{s_js_is_j^{-1}}s_k^{-1}s_j\\
&= s_js_ks_is_k^{-1}s_j\\
&= \langle t_jt_i \rangle^{3}
\end{align*}

\item We have $\langle t_it_j \rangle^{2} = s_i\underline{s_ks_js_k^{-1}} = \underline{s_is_j^{-1}s_ks_j} = \underline{s_j^{-1}s_ks_j}s_i = s_ks_js_k^{-1}s_i = \langle t_jt_i \rangle^{2}$
\end{enumerate}
\item
\begin{enumerate}[i)]
\item We have $\langle t_it_j \rangle^{2} = s_ks_is_k^{-1}s_ks_js_k^{-1} = s_ks_is_js_k^{-1} = s_ks_js_is_k^{-1} = \langle t_jt_i \rangle^{2}$
\item We have $\langle t_it_j \rangle^{2} = s_is_j = s_js_i = \langle t_jt_i \rangle^{2}$
\end{enumerate}
\item
\begin{enumerate}[i)]
\item We have
\begin{align*}
\langle t_it_j \rangle^{4}\langle t_jt_i \rangle^{-4} &= t_it_jt_it_jt_i^{-1}t_j^{-1}t_i^{-1}t_j^{-1}\\
 &= s_ks_is_k^{-1}s_js_ks_i\underline{s_k^{-1}s_js_k}s_i^{-1}\underline{s_k^{-1}s_j^{-1}s_k}s_i^{-1}s_k^{-1}s_j^{-1}\\
&=s_ks_is_k^{-1}\underline{s_js_ks_is_j}s_k\underline{s_j^{-1}s_i^{-1}s_j}s_k^{-1}s_j^{-1}s_i^{-1}s_k^{-1}s_j^{-1}\\
&=s_k\underline{s_is_k^{-1}s_ks_j}s_ks_is_ks_i^{-1}s_k^{-1}\underline{s_j^{-1}s_i^{-1}s_k^{-1}s_j^{-1}}\\
&=s_ks_j\underline{s_is_ks_is_ks_i^{-1}s_k^{-1}s_i^{-1}s_k^{-1}}s_j^{-1}s_k^{-1}\\
&= s_ks_j\langle s_is_j \rangle^{4}\langle s_js_i \rangle^{-4}s_j^{-1}s_k^{-1} \\
&= e
\end{align*}

\item We have $\langle t_it_j \rangle^{2} = s_is_ks_js_k^{-1} = s_is_j^{-1}s_ks_j = s_j^{-1}s_ks_js_i = s_ks_js_k^{-1}s_i = \langle t_jt_i \rangle^{2}$
\end{enumerate}
\item
\begin{enumerate}[i)]
\item We have
\begin{align*}
\langle t_it_j \rangle^{4}\langle t_jt_i \rangle^{-4} &= t_it_jt_it_jt_i^{-1}t_j^{-1}t_i^{-1}t_j^{-1}\\
&= \underline{s_ks_is_k^{-1}}s_j\underline{s_ks_is_k^{-1}}s_j\underline{s_ks_i^{-1}s_k^{-1}}s_j^{-1}\underline{s_ks_i^{-1}s_k^{-1}}s_j^{-1}\\
&= s_i^{-1}s_k\underline{s_is_js_i^{-1}}s_k\underline{s_is_js_i^{-1}}s_k^{-1}\underline{s_is_j^{-1}s_i^{-1}}s_k^{-1}\underline{s_is_j^{-1}}\\
&= s_i^{-1}\underline{s_ks_js_ks_js_k^{-1}s_j^{-1}s_k^{-1}s_j^{-1}}s_i\\
&= s_i^{-1}\langle t_kt_j \rangle^{4}\langle t_jt_k \rangle^{-4}s_i\\
&= e
\end{align*}

\item Since $s_js_k^{-1}s_is_k = s_k^{-1}s_is_ks_j$, we have that $s_js_k^{-1}s_i^{-1}s_k = s_k^{-1}s_i^{-1}s_ks_j$, hence

\begin{align*}
\langle t_it_j \rangle^{2}\langle t_jt_i \rangle^{-2} &= t_it_jt_i^{-1}t_j^{-1}\\
&= s_is_k\underline{s_js_k^{-1}s_i^{-1}s_k}s_j^{-1}s_k^{-1}\\
&= s_is_ks_k^{-1}s_i^{-1}s_ks_js_j^{-1}s_k^{-1}\\
&= e\\
\end{align*}

\end{enumerate}
\item
\begin{enumerate}[i)]
\item We have that
\begin{align*}
s_k^{-1}\langle t_it_j \rangle^{3}\langle t_jt_i \rangle^{-3}s_k &= s_k^{-1}t_it_jt_it_j^{-1}t_i^{-1}t_j^{-1}s_k\\
&= s_is_k^{-1}s_js_ks_is_k^{-1}s_j^{-1}s_ks_i^{-1}s_k^{-1}s_j^{-1}s_k\\
&= e
\end{align*}

Where the second equality follows from Lemma \ref{lem:extra_square_relation}.  It follows that $\langle t_it_j \rangle^{4} = \langle t_jt_i \rangle^{4}$.

\item This follows from part (i) by symmetry.
\end{enumerate}
\end{enumerate}
\end{proof}

\begin{lem}\label{lem:r3_relations}
The elements $t_i$ satisfy the $(T3^\prime)$ relations in $\Gamma^\prime$.
\end{lem}

\begin{proof}
We know that every chordless cycle in $\Gamma'$ arises from a subdiagram of $\Gamma$ in the form of one of the cases of Lemma \ref{lem:chordless_cycles}, so we simply need to check that a cycle relation holds in each case.  We follow the labeling of the vertices used in Lemma \ref{lem:chordless_cycles}.  We denote by $t^\prime(m,n)$ the expression given by replacing the generators $r_i$ in the corresponding expression $t(m,n)$ in $A_{\Gamma^\prime}$ with $t_i$.  When helpful we note uses of $(T3)$ relations in $A_\Gamma$ by referencing the particular cycle relation used next to the manipulation.

\begin{enumerate}[a)]
\item We have $t_1t_2^{-1}t_3t_2 = s_1s_2^{-1}s_2s_3s_2^{-1}s_2 = s_1s_3 = s_3s_1 = t_2^{-1}t_3t_2t_1$, hence $t^\prime(1,2) = e$.

\item We have $ t_1t_2^{-1}t_3t_2 = s_1s_2^{-1}s_2s_3s_2^{-1}s_2 = s_1s_3 = s_3s_1 = t_2^{-1}t_3t_2t_1$, hence $t^\prime(1,2) = e$.

\item We have $t_1t_2^{-1}t_3t_2 = s_1s_2^{-1}s_2s_3s_2^{-1}s_2 = s_3s_1 = t_2^{-1}t_3t_2t_1$, hence $t^\prime(1,2) = e$.

\item We have
\begin{align*}
s_2^{-1}t^\prime(3,1)s_2 &= s_2^{-1}t_3t_1^{-1}t_2t_1t_3^{-1}t_1^{-1}s_2^{-1}t_1s_2\\
&= s_3\underline{s_2^{-1}s_1^{-1}s_2s_1s_2}s_3^{-1}\underline{s_2^{-1}s_1^{-1}s_2^{-1}s_1s_2}\\
&= s_3s_1\underline{s_2s_1^{-1}s_3^{-1}s_1}s_2^{-1}s_1^{-1}\\
&= s_3s_1s_1^{-1}s_3^{-1}s_1s_2s_2^{-1}s_1^{-1} &\text{by } t(2,1)\\
&= e
\end{align*}

\noindent Hence $t^\prime(3,1) = e$.

\item We have
\begin{align*}
t_1t_2^{-1}t_3^{-1}t_4t_3t_2 &= (s_1s_2^{-1}s_1^{-1}s_1s_2s_1^{-1})s_1s_1s_2^{-1}s_1^{-1}s_3^{-1}s_4s_3s_1s_2s_1^{-1}\\
&= s_1s_2^{-1}s_1^{-1}\underline{s_1s_2s_1s_2^{-1}s_1^{-1}}s_3^{-1}s_4s_3s_1s_2s_1^{-1}\\
&= s_1s_2^{-1}s_1^{-1}\underline{s_2s_3^{-1}s_4s_3}s_1s_2s_1^{-1}\\
&= s_1s_2^{-1}s_1^{-1}s_3^{-1}s_4s_3\underline{s_2s_1s_2s_1^{-1}} &\text{by } t(2,3)\\
&= s_1s_2^{-1}s_1^{-1}s_3^{-1}s_4s_3s_1s_2s_1^{-1}s_1\\
&= t_2^{-1}t_3^{-1}t_4t_3t_2t_1\\
\end{align*}

\noindent Hence $t^\prime(1,2) = e$.

\item We have
\begin{align*}
t_1t_2^{-1}t_3^{-1}t_4t_3t_2 &= s_1\underline{s_2^{-1}s_2}s_3^{-1}\underline{s_2^{-1}s_4s_2}s_3 \underline{s_2^{-1}s_2}\\
&= s_1s_3^{-1}s_4s_3\\
&= s_3^{-1}s_4s_3s_1 &\text{by } t(1,3)\\
&= (s_2^{-1}s_2)s_3^{-1}(s_2^{-1}\underline{s_2)s_4(s_2^{-1}}s_2)s_3(s_2^{-1}s_2s_1\\
&= s_2^{-1}s_2s_3^{-1}s_2^{-1}s_4s_2s_3s_2^{-1}s_2s_1\\
&= t_2^{-1}t_3^{-1}t_4t_3t_2t_1\\
\end{align*}

\noindent Hence $t^\prime(1,2) = e$.

\item We have
\begin{align*}
t^\prime(3,4) &= t_3t_4^{-1}t_2t_4t_3^{-1}t_4^{-1}t_2^{-1}t_4\\
&= \underline{s_3s_1}s_4^{-1}s_1^{-1}s_2s_1s_4\underline{s_1^{-1}s_3^{-1}s_1}s_4^{-1}s_1^{-1}s_2^{-1}s_1s_4s_1^{-1}\\
&= s_1\underline{s_3s_4^{-1}s_1^{-1}s_2s_1s_4s_3^{-1}s_4^{-1}s_1^{-1}s_2^{-1}s_1s_4}s_1^{-1}\\
&= s_1t(3,4)s_1^{-1}\\
&= s_1s_1^{-1}\\
&= e\\
\end{align*}

\item We have
\begin{align*}
t_3t_4^{-1}t_1t_4 &= s_3s_4^{-1}\underline{s_2s_1s_2^{-1}}s_4\\
&= s_3s_4^{-1}s_1^{-1}s_2s_1s_4\\
&= s_4^{-1}\underline{s_1^{-1}s_2s_1}s_4s_3 &\text{by} t(3,4)\\
&= s_4^{-1}s_2s_1s_2^{-1}s_4s_3\\
&= t_4^{-1}t_1t_4t_3\\
\end{align*}

\noindent Hence $t^\prime(3,4) = e$.

\item We have
\begin{align*}
&t^\prime(1,2)\\
&= t_1t_2^{-1}t_3^{-1}\cdots t_{h-1}^{-1}t_ht_{h-1}\cdots t_2t_1^{-1}t_2^{-1}\cdots t_{h-1}^{-1}t_h^{-1}t_{h-1}\cdots t_2\\
&= s_1s_2^{-1}s_3^{-1}\cdots s_{h-1}^{-1}\underline{s_ks_hs_k^{-1}}s_{h-1}\cdots s_2s_1^{-1}s_2^{-1}\cdots s_{h-1}^{-1}\underline{s_ks_h^{-1}s_k^{-1}}s_{h-1}\cdots s_2\\
&= s_1s_2^{-1}s_3^{-1}\cdots s_{h-1}^{-1}s_h^{-1}s_ks_hs_{h-1}\cdots s_2s_1^{-1}s_2^{-1}\cdots s_{h-1}^{-1}s_h^{-1}s_k^{-1}s_hs_{h-1}\cdots s_2\\
&= t(1,2)\\
&= e\\
\end{align*}

\item Here $s_k$ commutes with $s_i$ for all $i\ne 1,h$, so we have
\begin{align*}
& t^\prime(h,1)\\
&= t_ht_k^{-1}t_1^{-1}t_2^{-1}\cdots t_{h-2}^{-1}t_{h-1}t_{h-2}\cdots t_1t_kt_h^{-1}t_k^{-1}t_1^{-1}t_2^{-1}\cdots t_{h-2}^{-1}t_{h-1}^{-1}t_{h-2}\cdots t_2t_1t_k\\
&= s_h\underline{s_k^{-1}s_k}s_1^{-1}\underline{s_k^{-1}}s_2^{-1}\cdots s_{h-2}^{-1}s_{h-1}s_{h-2}\cdots \underline{s_k}s_1\underline{s_k^{-1}s_k}s_h^{-1}\underline{s_k^{-1}s_k}s_1^{-1} \\
&\hspace{2cm}\underline{s_k^{-1}}s_2^{-1}\cdots s_{h-2}^{-1}s_{h-1}^{-1}s_{h-2}\cdots s_2 \underline{s_k}s_1\underline{s_k^{-1}s_k}\\
&= s_hs_1^{-1}s_2^{-1}\cdots s_{h-2}^{-1}s_{h-1}s_{h-2}\cdots s_2s_1s_h^{-1}s_1^{-1}s_2^{-1}\cdots s_{h-2}^{-1}s_{h-1}^{-1}s_{h-2}\cdots s_2s_1\\
&= t(h,1)\\
&= e\\
\end{align*}

\item Here $t_i = s_i$ for all vertices $i$ in $C^\prime$, so the case is trivial.

\item  If the edge between $k$ and $h$ points towards $h$, then $t_i = s_i$ for all $i\in C^\prime$ and the case is trivial.  If the edge points towards $k$ then $t_h = s_ks_hs_k^{-1}$ and $s_k$ commutes with $s_i$ for all vertices $i\ne h$ in $C^\prime$, so we have

\begin{align*}
t_1t_2^{-1}\cdots t_{h-1}^{-1}t_ht_{h-1}\cdots t_2 &= s_1s_2^{-1}\cdots s_{h-1}^{-1}s_ks_hs_k^{-1}s_{h-1}\cdots s_2\\
&= s_k\underline{s_1s_2^{-1}\cdots s_{h-1}^{-1}s_hs_{h-1}\cdots s_2}s_k^{-1}\\
&= s_ks_2^{-1}\cdots s_{h-1}^{-1}s_hs_{h-1}\cdots s_2s_1s_k^{-1} &\text{by } t(1,2)\\
&= s_2^{-1}\cdots s_{h-1}^{-1}s_ks_hs_k^{-1}s_{h-1}\cdots s_2s_1\\
&= t_2^{-1}\cdots t_{h-1}^{-1}t_ht_{h-1}\cdots t_2t_1\\
\end{align*}

\noindent Hence $t^\prime(1,2) = e$.
\end{enumerate}
\end{proof}

\begin{proof}[Proof of Proposition \ref{prop:ti_relations}]
Lemma \ref{lem:r2a_relations} and Lemma \ref{lem:r2b_relations} show that the elements $t_i$ satisfy the $(T2^\prime)$ relations for $A_{\Gamma^\prime}$. Lemma \ref{lem:r3_relations} shows that they satisfy the $(T3^\prime)$ relations. Since these are all of the relations defining $A_{\Gamma^\prime}$, it follows that $\varphi$ defines a group homomorphism $A_{\Gamma^\prime}\rightarrow A_\Gamma$.
\end{proof}

\section{Extension to Affine Type Diagrams} \label{sec:affine}
It is natural to ask whether one can obtain similar results for diagrams that are not of finite type, particularly those of affine type (i.e. mutation equivalent to an affine Dynkin diagram). In \cite{FT13}, the authors associated the following group to a diagram of affine type.

\begin{def}\cite[Definition 4.1]{FT13}
Let $\Gamma$ be a diagram of affine type with n + 1 vertices. Then we define $W_{\Gamma}$ to be the diagram with generators $s_{1}, \dots, s_{n}$ and satisfying the following relations:
\begin{enumerate}
\item[(R1)] $s_{i}^{2} = e$ for all i = 1, \dots, n
\item[(R2)] $(s_{i}s_{j})^{m_{ij}} = e$ where 
$$m_{ij} = 
\begin{cases}
2 &\text{if there is no arrow between i and j in $\Gamma$} \\
3 &\text{if there is an arrow of weight 1 between i and j in $\Gamma$} \\
4 &\text{if there is an arrow of weight 2 between i and j in $\Gamma$} \\
6 &\text{if there is an arrow of weight 3 between i and j in $\Gamma$} \\
\infty &\text{otherwise}
\end{cases}$$
\item[(R3)] For every chordless oriented cycle:
$$i_{0} \stackrel{w_{i_{0}}}{\longrightarrow} i_{1} \stackrel{w_{i_{1}}}{\longrightarrow} \dots \stackrel{w_{i_{d-2}}}{\longrightarrow} i_{d-1} \stackrel{w_{i_{d-1}}}{\longrightarrow} i_{0},$$
define for l $\in \{0, \dots, d-1\}$, 
$$t(l) = (\prod_{j=l}^{l+d-2}{\sqrt{w_{i_{j}}}} - \sqrt{w_{i_{l+d-1}}})^{2}.$$
Then take the relation $(s_{i_{l}}p(i_{l}, i_{l+1}))^{m(l)} = e$ where
$$m(l) =
\begin{cases}
2 &\text{if $t(l)=0$} \\
3 &\text{if $t(l)=1$} \\
4 &\text{if $t(l)=2$} \\
6 &\text{if $t(l)=3$}
\end{cases}$$
\item[(R4)] For each subdiagram of $\Gamma$ of the form shown in the first column of Table \ref{table1}, we add the relation(s) listed in the second column.
\end{enumerate}
\end{def}

Given this definition of $W_{\Gamma}$, which generalizes the definition of $W_{\Gamma}$ found in \cite{BM13}, one obtains the following result.

\begin{thm} \cite[Theorem 4.6]{FT13}
Let W be an affine Weyl group and let $\Gamma$ be a diagram mutation equivalent to an orientation of a Dynkin diagram of the same type as W. Then W is isomorphic to $W_{\Gamma}$.
\end{thm}

\begin{table}[h] 
\begin{tabular}{| p{3.5cm} | p{7cm} |}
\hline
Subdiagram & (R4) Relation \\ \hline
\begin{center}\includegraphics[scale = .30]{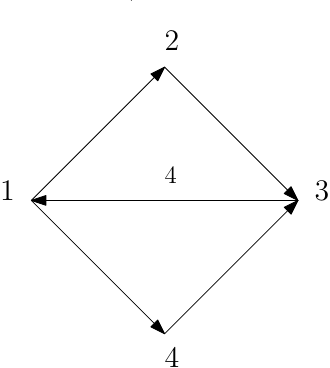}\end{center} & $(s_{1}s_{2}s_{3}s_{4}s_{3}s_{2})^{2} = e$ \\ \hline

\begin{center}\includegraphics[scale = .30]{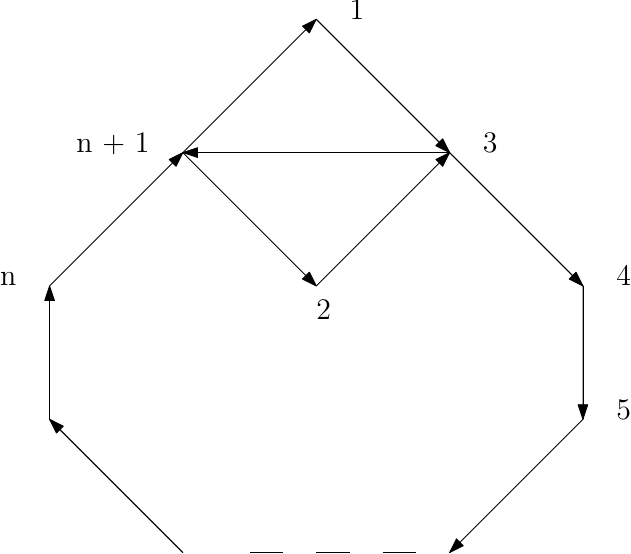}\end{center} 
& $(s_{1}s_{2}s_{3}s_{2}s_{1}s_{4}s_{5} \dots s_{n}s_{n+1}s_{n} \dots s_{5}s_{4})^{2} = e$ \\ \hline

\begin{center}\includegraphics[scale = .30]{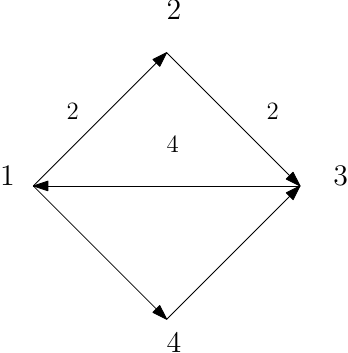}\end{center} & $(s_{2}s_{3}s_{4}s_{1}s_{4}s_{3})^{2} = e$ and $(s_{2}s_{1}s_{4}s_{3}s_{4}s_{1})^{2} = e$\\ \hline

\begin{center}\includegraphics[scale = .30]{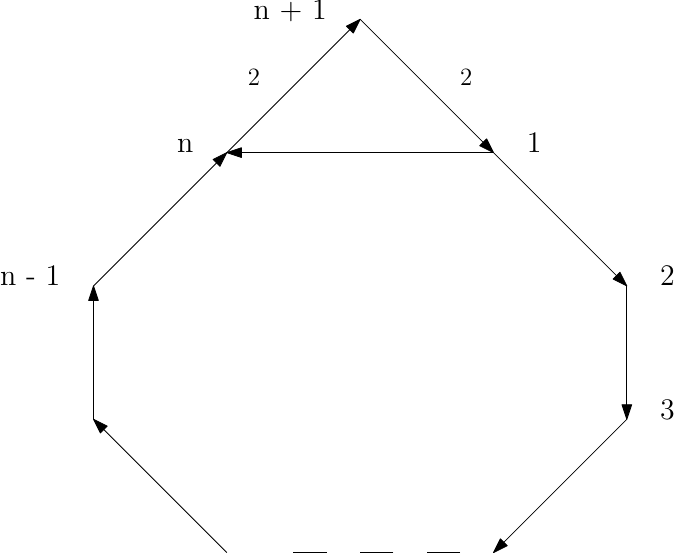}\end{center} & $(s_{n+1}s_{1}s_{n+1}s_{2}s_{3} \dots s_{n-1}s_{n}s_{n-1} \dots s_{3}s_{2})^{2} = e$ \\ \hline

\begin{center}\includegraphics[scale = .30]{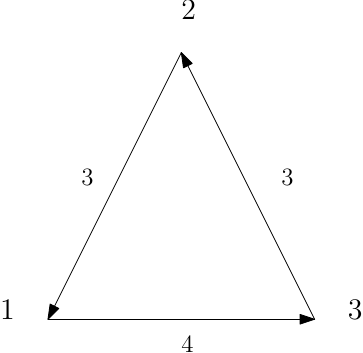}\end{center} & $(s_{2}s_{1}s_{2}s_{1}s_{2}s_{3})^{2} = e$ and $(s_{2}s_{3}s_{2}s_{3}s_{2}s_{1})^{2} = e$ \\ \hline
\end{tabular}
\caption{(R4) Relations} \label{table1}
\end{table}

In seeking an analogous result for Artin groups, we generalize A$_{\Gamma}$ as follows:

\begin{defn}
Let $\Gamma$ be a diagram of affine type with n + 1 vertices. Then we define $A_{\Gamma}$ to be the diagram with generators $s_{1}, \dots, s_{n}$ and satisfying the following relations:
\begin{enumerate}
\item[(T2)] $\langle s_{i}, s_{j} \rangle^{m_{ij}} = \langle s_{j}, s_{i} \rangle^{m_{ij}}$ where
$$m_{ij} = 
\begin{cases}
2 &\text{if there is no arrow between i and j in $\Gamma$} \\
3 &\text{if there is an arrow of weight 1 between i and j in $\Gamma$} \\
4 &\text{if there is an arrow of weight 2 between i and j in $\Gamma$} \\
6 &\text{if there is an arrow of weight 3 between i and j in $\Gamma$} \\
\infty &\text{otherwise}
\end{cases}$$
\item[(T3)] For every chordless oriented cycle:
$$i_{0} \stackrel{w_{i_{0}}}{\longrightarrow} i_{1} \stackrel{w_{i_{1}}}{\longrightarrow} \dots \stackrel{w_{i_{d-2}}}{\longrightarrow} i_{d-1} \stackrel{w_{i_{d-1}}}{\longrightarrow} i_{0},$$
define for l $\in \{0, \dots, d-1\}$, 
$$t(l) = (\prod_{j=l}^{l+d-2}{\sqrt{w_{i_{j}}}} - \sqrt{w_{i_{l+d-1}}})^{2}.$$
Then take the relation $\langle s_{i_{l}}p(i_{l}, i_{l+1}) \rangle^{m(l)} = \langle p(i_{l}, i_{l+1})s_{i_{l}} \rangle^{m(l)}$ where
$$m(l) =
\begin{cases}
2 &\text{if $t(l)=0$} \\
3 &\text{if $t(l)=1$} \\
4 &\text{if $t(l)=2$} \\
6 &\text{if $t(l)=3$}
\end{cases}$$
\item[(T4)] For each subdiagram of $\Gamma$ of the form shown in the first column of Table \ref{table2}, we add the relation(s) listed in the second column.
\end{enumerate}
\end{defn}

\begin{conj}
Let $\Gamma$ be a diagram of affine type, and let $\Gamma^{\prime}$ be the mutation of $\Gamma$ at a node k. Then $A_{\Gamma}$ $\cong$ $A_{\Gamma^{\prime}}$.
\end{conj}

\begin{table}[h] 
\begin{tabular}{| p{3.5cm} | p{7cm} |}
\hline
Subdiagram & (T4) Relation \\ \hline
\begin{center}\includegraphics[scale = .30]{Diagram1.pdf}\end{center} & $s_{2}s_{1}s_{2}^{-1}s_{4}^{-1}s_{3}s_{4} = s_{4}^{-1}s_{3}s_{4}s_{2}s_{1}s_{2}^{-1}$ \\ \hline

\begin{center}\includegraphics[scale = .30]{Diagram2.pdf}\end{center} 
& $s_{2}s_{3}^{-1}s_{4}^{-1}\dots s_{n}^{-1}s_{1}s_{n+1}s_{1}^{-1}s_{n} \dots s_{4}s_{3}$
\newline $= s_{3}^{-1}s_{4}^{-1}\dots s_{n}^{-1}s_{1}s_{n+1}s_{1}^{-1}s_{n} \dots s_{4}s_{3}s_{2}$\\ \hline

\begin{center}\includegraphics[scale = .30]{Diagram3.pdf}\end{center} & $s_{2}s_{3}^{-1}s_{4}s_{1}s_{4}^{-1}s_{3} = s_{3}^{-1}s_{4}s_{1}s_{4}^{-1}s_{3}s_{2}$ and
$s_{2}s_{1}s_{4}^{-1}s_{3}s_{4}s_{1}^{-1} = s_{1}s_{4}^{-1}s_{3}s_{4}s_{1}^{-1}s_{2}$\\ \hline

\begin{center}\includegraphics[scale = .30]{Diagram4.pdf}\end{center} & $s_{1}s_{2}^{-1} \dots s_{n-1}^{-1}s_{n+1}s_{n}s_{n+1}^{-1}s_{n-1} \dots s_{2} = s_{2}^{-1} \dots s_{n-1}^{-1}s_{n+1}s_{n}s_{n+1}^{-1}s_{n-1} \dots s_{2}s_{1}$ \\ \hline

\begin{center}\includegraphics[scale = .30]{Diagram5.pdf}\end{center} & $s_{2}s_{1}^{-1}s_{2}s_{3}s_{2}^{-1}s_{1} = s_{1}^{-1}s_{2}s_{3}s_{2}^{-1}s_{1}s_{2}$ and $s_{1}s_{2}s_{3}s_{2}s_{3}^{-1}s_{2}^{-1} = s_{2}s_{3}s_{2}s_{3}^{-1}s_{2}^{-1}s_{1}$ \\ \hline
\end{tabular}
\caption{(T4) Relations} \label{table2}
\end{table}

\section{Acknowledgements}
This work was supported by the RTG grant NSF/DMS-1148634. We would also like to thank the University of Minnesota School of Mathematics for hosting us. We are greatly appreciative of Gregg Musiker's mentoring, support, and guidance on this problem. We thank Emily Gunawan for her useful comments on our work. We also thank Vic Reiner for motivating this problem.


\vspace{0.5cm}

\begin{thebibliography}{Cha06}

\bibitem[BM13]{BM13}
Michael Barot and Robert~J. Marsh.
\newblock Reflection group presentations arising from cluster algebras.
\newblock arXiv: 1112.2300v2, 2013.

\bibitem[Cha06]{C06}
Ruth Charney.
\newblock An introduction to right-angled artin groups.
\newblock arXiv:math/0610668, 2006.

\bibitem[FN61]{FN61}
R.~Fox and L.~Neuwirth.
\newblock The braid groups.
\newblock {\em Mathematica Scandinavica}, pages 119--126, 1961.

\bibitem[FT13]{FT13}
A.~Felikson and P.~Tumarkin.
\newblock Coxeter groups and their quotients arising from cluster algebras.
\newblock arXiv:1307.0672v2, 2013.

\bibitem[FZ02]{FZ02}
Sergey Fomin and Andrei Zelevinsky.
\newblock Cluster algebras I: Foundations.
\newblock {\em J. Amer. Math. Soc.}, 15:497--529, 2002.

\bibitem[FZ03]{FZ03}
Sergey Fomin and Andrei Zelevinsky.
\newblock Cluster algebras II: Finite type classification.
\newblock {\em Invent. Math.}, 154:63--121, 2003.

\bibitem[GM14]{GM14}
Joseph Grant and Robert~J. Marsh.
\newblock Braid groups and quiver mutation.
\newblock arXiv:1408.5267, 2014.

\bibitem[Qiu14]{Q14}
Y.~Qiu.
\newblock On the spherical twists on 3-calabi-yau categories from marked
  surfaces.
\newblock arXiv:1407.0806, 2014.

\end{thebibliography}
\end{document}